\documentclass[bj,numbers,noshowframe]{imsart}

\usepackage{amsfonts}
\usepackage{amsmath}
\usepackage{amssymb}
\usepackage{amsthm}
\usepackage{bbm}
\usepackage{bm}
\usepackage{comment}
\usepackage{mathrsfs}
\usepackage{graphicx}

\startlocaldefs
\theoremstyle{plain}

\newtheorem{theorem}{Theorem}[section]
\newtheorem{lemma}[theorem]{Lemma}
\newtheorem{remark}[theorem]{Remark}
\newtheorem{proposition}[theorem]{Proposition}
\theoremstyle{remark}
\newtheorem{definition}[theorem]{Definition}
\newtheorem*{example}{Example}

\newcommand{\var}[1]{{\rm Var}\left(#1\right)}

\newcommand{\OT}{\mathsf{OT}}
\newcommand{\bs}{\boldsymbol}

\DeclareMathOperator*{\argmin}{arg\, min}

\newcommand{\cA}{\mathcal{A}}

\newcommand{\cF}{\mathcal{F}}

\newcommand{\pa}{\partial}
\newcommand{\h}{\hat}

\endlocaldefs
\begin{document}

\begin{frontmatter}
\title{Stability and sample complexity of divergence regularized optimal transport}
\runtitle{Stability and sample complexity of divergence regularized optimal transport}

\begin{aug}
\author[A]{\inits{E.}\fnms{Erhan}~\snm{Bayraktar}\ead[label=e1]{erhan@umich.edu}}
\author[B]{\inits{S.}\fnms{Stephan}~\snm{Eckstein}\ead[label=e2]{stephan.eckstein@uni-tuebingen.de}\thanks{corresponding author}}
\author[C]{\inits{X.}\fnms{Xin}~\snm{Zhang}\ead[label=e3]{xin.zhang@univie.ac.at}}
\address[A]{Department of Mathematics, University of Michigan, Michigan, USA\printead[presep={,\ }]{e1}}

\address[B]{Department of Mathematics, University of T\"{u}bingen, T\"{u}bingen, Germany\printead[presep={,\ }]{e2}}

\address[C]{Department of Mathematics, University of Vienna, Vienna, Austria\printead[presep={,\ }]{e3}}
\end{aug}

\begin{abstract}
We study stability and sample complexity properties of divergence regularized optimal transport (DOT). First, we obtain quantitative stability results for optimizers of DOT measured in Wasserstein distance, which are applicable to a wide class of divergences and simultaneously improve known results for entropic optimal transport. Second, we study the case of sample complexity, where the DOT problem is approximated using empirical measures of the marginals.
We show that divergence regularization can improve the corresponding convergence rate compared to unregularized optimal transport. To this end, we prove upper bounds which exploit both the regularity of cost function and divergence functional, as well as the intrinsic dimension of the marginals.
Along the way, we establish regularity properties of dual optimizers of DOT, as well as general limit theorems for empirical measures with suitable classes of test functions.
\end{abstract}


\end{frontmatter}

\section{Introduction}
Regularization of optimal transport problems introduces various benefits, mainly in terms of computational tractability (see, e.g., \cite{blanchet2018towards,cuturi2013sinkhorn,lin2019efficient,peyre2019computational,semi-discrete}). The most well-known regularization method is the entropic one, leading to the so-called entropic optimal transport (EOT) problem, allowing for the use of Sinkhorn's algorithm (see \cite{cuturi2013sinkhorn,franklin1989scaling, sinkhorn1967concerning}). Beyond its computational aspects, EOT enjoys other beneficial properties
compared to unregularized OT, such as stability of optimizers with respect to perturbations of marginals or cost function (see, e.g., \cite{carlier2020differential,deligiannidis2021quantitative,eckstein2022quantitative,ghosal2022stability,keriven2022entropic,nutz2022stability}), and improved estimation rates when only finitely many sample points for each marginal are known (see, e.g., \cite{genevay2019sample, klatt2020empirical, mena2019statistical, rigollet2022sample}). While different aspects, like numerical stability and sparseness of optimal couplings, motivate the use of other divergences aside from the entropic one (see, e.g., \cite{blondel2018smooth,eckstein2021computation, essid2018quadratically,lorenz2021quadratically,seguy2017large}), similar results regarding stability and sample complexity have not yet been established beyond the entropic case. This paper aims to fill this gap.

\subsection{Overview of the main results}
We first introduce the optimal transport problem with divergence regularization in Section \ref{subsubsec:drot} before giving an overview of the main results of this paper in Sections \ref{subsubsec:quantstab} and \ref{subsubsec:samplecomp}.

\subsubsection{Divergence regularized optimal transport}
\label{subsubsec:drot}
The divergence regularized optimal transport (DOT) problem takes the form
\begin{equation}\label{eq:DOTintro}
\OT^\varepsilon_\varphi(\mu_1, \mu_2) := \inf_{\pi \in \Pi(\mu_1, \mu_2)} \int c \,d\pi + \varepsilon D_\varphi(\pi, \mu_1 \otimes \mu_2),
\end{equation}
where $\mu_1, \mu_2$ are the marginals, $\Pi(\mu_1, \mu_2)$ is the set of couplings between $\mu_1$ and $\mu_2$, $c$ is the cost function, $D_\varphi$ is the divergence functional and $\varepsilon > 0$ is the regularization parameter; the precise notation is detailed in Section \ref{sec:notation}. While we keep the exposition simple in the introduction, all results will also apply to the version of DOT with multiple marginals $\mu_1, \dots, \mu_N$.

Regarding the motivation for the DOT problem, we shortly want to illustrate some key aspects of this problem with a focus on the differences between varying divergence functionals $D_\varphi$. In general, under strict convexity of $\varphi$, the addition of the term $\varepsilon D_\varphi(\pi, \mu_1 \otimes \mu_2)$ transforms the optimization from a linear to a strictly convex problem.
A key feature of DOT is that the optimizer $\pi^*$ of $\OT^\varepsilon_\varphi(\mu_1, \mu_2)$ is then unique and characterized by its density 
	\begin{equation}\label{eq:density}
	\frac{d\pi^*}{d\mu_1 \otimes \mu_2}(x_1, x_2) = \psi'((h_1^*(x_1)+h_2^*(x_2)-c(x_1, c_2))/\varepsilon),
	\end{equation}
where $\psi$ is the convex conjugate of $\varphi$ and $h_1^*$ and $h_2^*$ are dual optimizers (cf.~Lemmas \ref{lem:primalexistence} and \ref{lemma:dualoptimizer}, and \cite[Theorem 2.2]{eckstein2021computation}).
For the entropic optimal transport problem, $\varphi(x) = x\log(x)$, we have $\psi'(y) = \exp(y-1)$, and thus optimizers will have the same support as the product measure. I.e., while the optimizers will be very smooth in the sense of having a smooth density, they will never be sparse! This is in stark contrast to optimizers of the unregularized ($\varepsilon = 0$) problem, where optimizers are usually supported on a low-dimensional manifold (cf.~\cite{mccann2012rectifiability}). For polynomial divergences like $\varphi(x) = \frac{1}{\alpha}(x^{\alpha}-1)$, we have $\psi'(y) = \max\{0, y\}^{\beta-1}$, where $\frac{1}{\alpha}+\frac{1}{\beta} = 1$, allowing for sparseness. This means that polynomial regularization leads to optimizers which are less smooth than optimizers for entropic regularization, but which can recover some key features of the unregularized problem (like sparseness, cf.~\cite{blondel2018smooth,dessein2018regularized}). We illustrate the difference between various regularization methods in a simple example in Figure \ref{fig:support}.
\begin{figure}\label{fig:support}
	\includegraphics[width=0.32\textwidth]{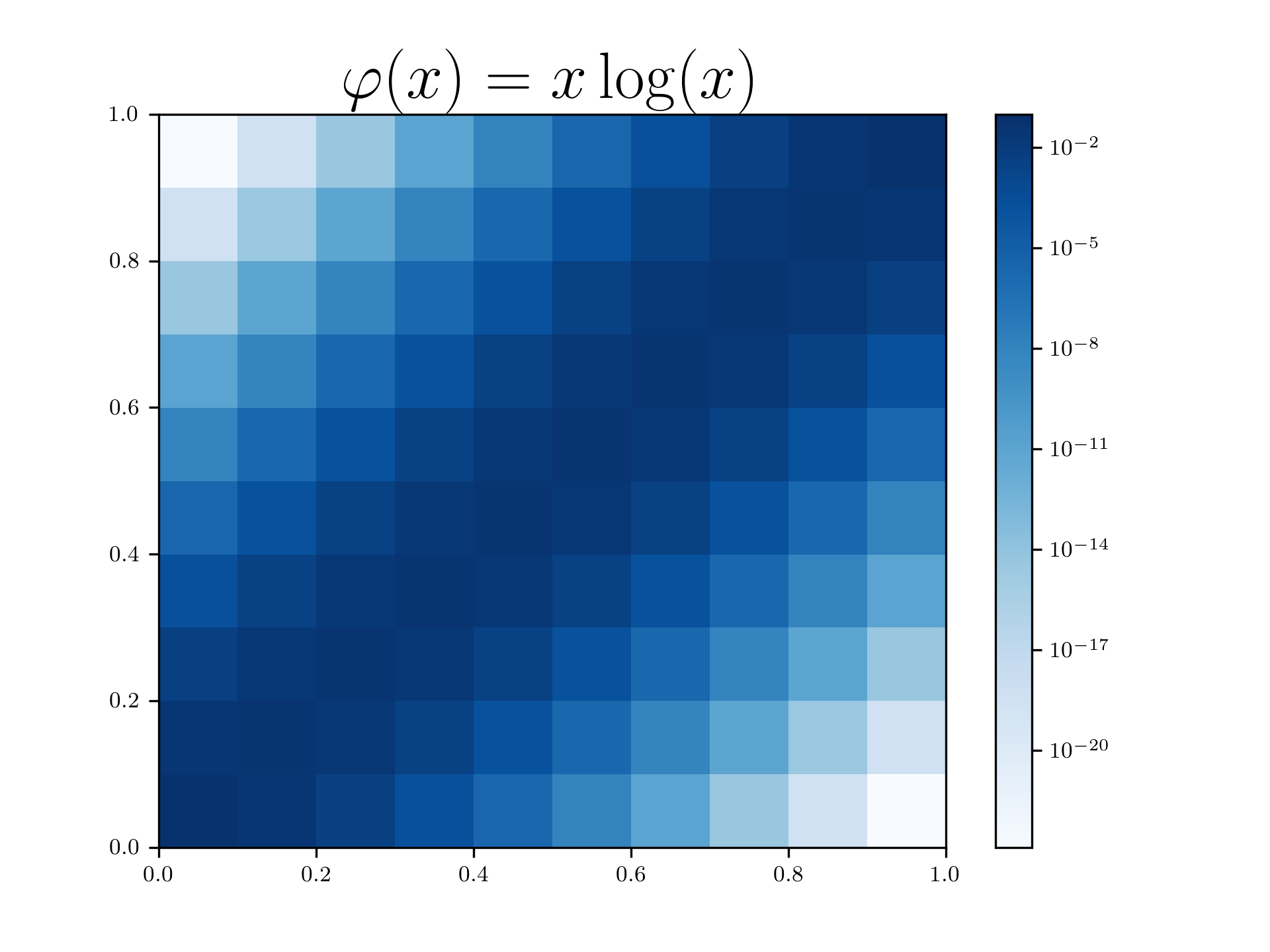}
	\includegraphics[width=0.32\textwidth]{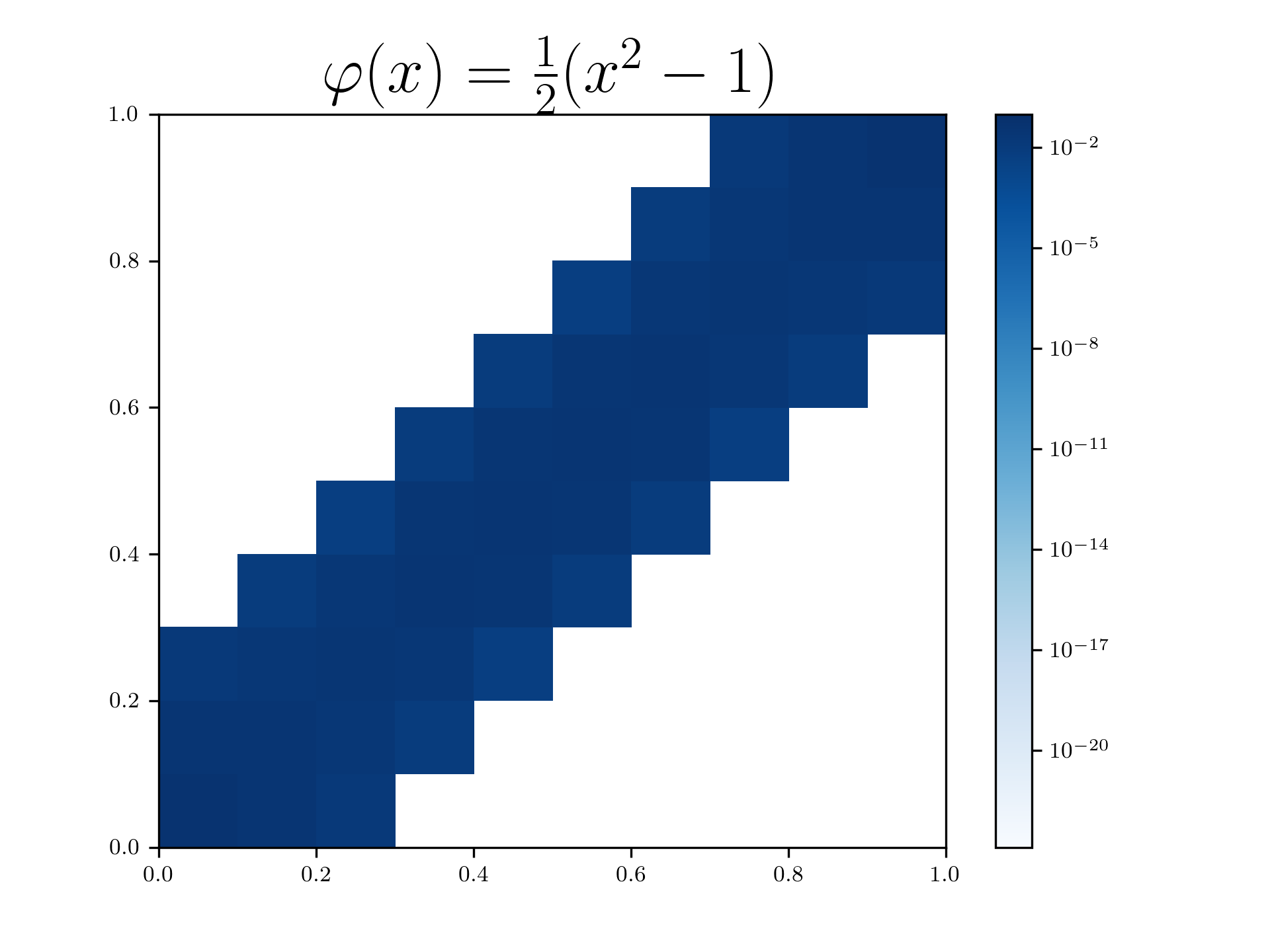}
	\includegraphics[width=0.32\textwidth]{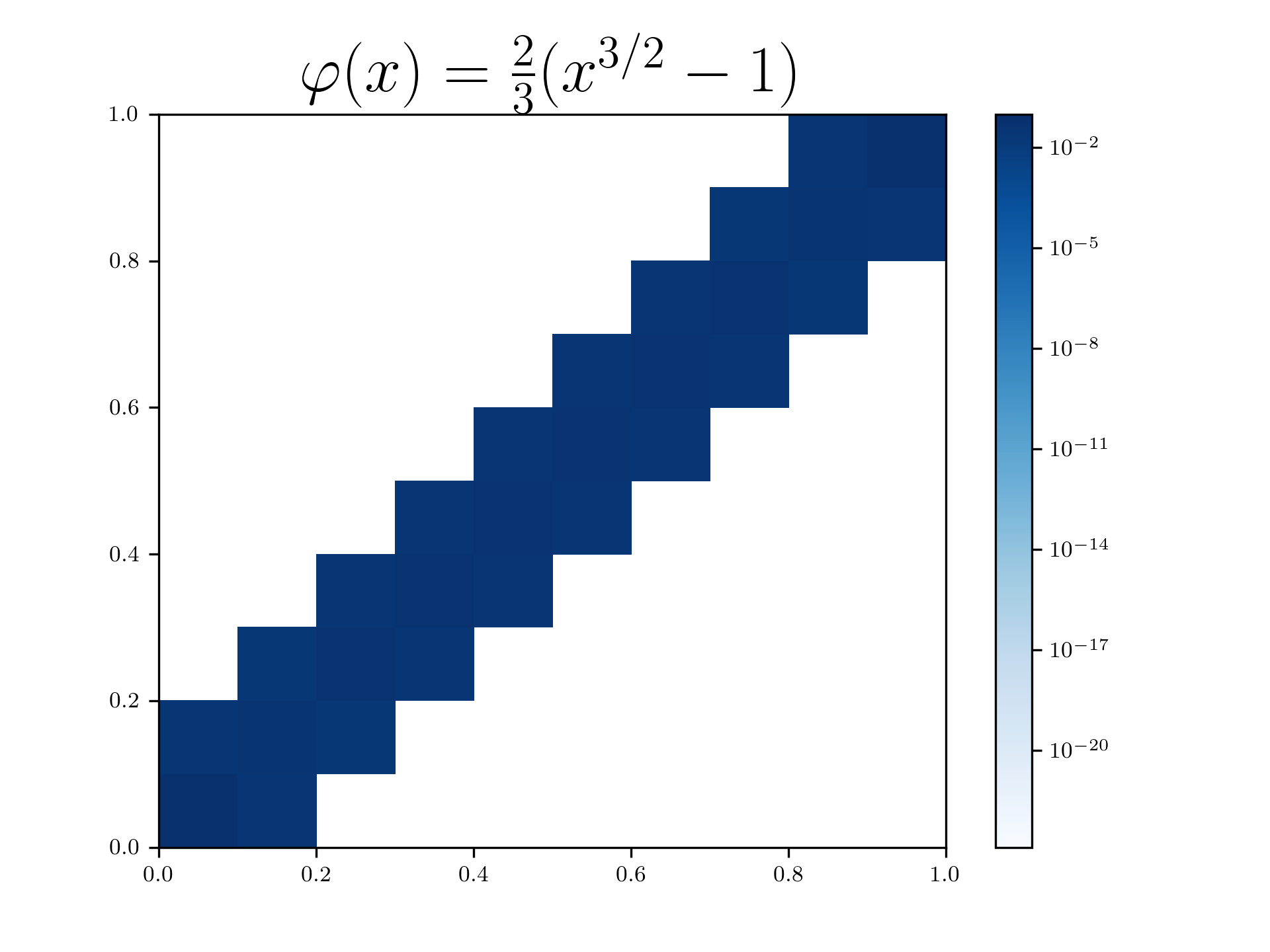}
	\caption{A log-scale visualization of optimizers $\pi^*$ of $\OT^\varepsilon_\varphi(\mu_1, \mu_2)$ for different $\varphi$ defining the regularization divergence $D_\varphi$. This example illustrates that structural properties of $\pi^*$, like the size of the support, can vary greatly depending on the type of regularization. From left to right, the respective sizes of the support of $\pi^*$ are 100 (full support), 44 and 28. The problem parameters are $\varepsilon=100$, $\mu_1 = \mu_2 = \frac{1}{10}\sum_{i=0}^9 \delta_{i/9}$ and $c(x_1, x_2) = (x_2-x_1)^2$.}
\end{figure}

This structural difference of the optimizers given by \eqref{eq:density} is particularly important for small $\varepsilon$. In this regard entropic regularization is often difficult to deal with since the values in the density \eqref{eq:density} scale exponentially in $\varepsilon$, leading for instance to instabilities in algorithmic methods (cf.~\cite{schmitzer2019stabilized}). Particularly for dual solution methods based on neural networks (cf.~\cite{eckstein2021computation, seguy2017large}), polynomial regularization can therefore be advantageous, since optimization methods for the neural network parameters based on gradient descent are less ill-behaved for small $\varepsilon$.

\subsubsection{Quantitative stability of optimizers}
\label{subsubsec:quantstab}
Regarding the obtained results, we first discuss the quantitative stability results for $\OT_{\varphi}^\varepsilon(\mu_1, \mu_2)$. While continuity of the optimal values, i.e., of the mapping $(\mu_1, \mu_2) \mapsto \OT_{\varphi}^\varepsilon(\mu_1, \mu_2)$ is already treated in \cite{eckstein2022quantitative}, this paper establishes continuity of the mapping
\begin{align}\label{eq:introstab}
(\mu_1, \mu_2) &\mapsto \pi^* := \argmin_{\pi \in \Pi(\mu_1, \mu_2)} \int c \,d\pi + \varepsilon D_\varphi(\pi, \mu_1 \otimes \mu_2).
\end{align}
The main result in this regard is given in Theorem \ref{thm:quantstability}, whereby differences in both marginals and optimizers are measured in a Wasserstein distance. Assuming weakened versions of Lipschitzianity for $c$ (see Definition \ref{def:lipschitz}) and strong convexity for $\varphi$ (see Definition \ref{def:strongconvex}), Theorem \ref{thm:quantstability} shows that the mapping in \eqref{eq:introstab} is H\"{o}lder continuous of order $\frac{1}{2p}$ when measuring distances in the $p$-th Wasserstein distance and restricting to marginals with uniformly bounded $2p$-th moments. The method of proof follows the line of \cite[Theorem 3.11]{eckstein2022quantitative}, where a similar result for EOT is established. There are, however, some notable differences. The proof of \cite[Theorem 3.11]{eckstein2022quantitative} relies on the Pythagorean theorem for relative entropy (c.f.~\cite{csiszar1975divergence}) and a subsequent application of a generalized Pinsker inequality \cite{bolley2005weighted}. While tools for the latter step are available also for non-entropic divergences (e.g., \cite{agrawal2021optimal}), there is no suitable counterpart for the Pythagorean theorem for general divergences. In the entropic case, combining these tools establishes a strong convexity property for the functional
\[
\Pi(\mu_1, \mu_2) \ni \pi \mapsto \int c \,d\pi + \varepsilon D_\varphi(\pi, \mu_1\otimes \mu_2)
\] 
around its optimizer.
We circumvent the use of a Pythagorean-like theorem and directly establish a strong convexity property for this functional in Lemma \ref{lem:strongconvexity} for a general class of divergences. Interestingly, the resulting estimate in Theorem \ref{thm:quantstability} requires weaker moment assumptions than \cite[Theorem 3.11]{eckstein2022quantitative}, even though it is also applicable to the EOT problem. The reason is that the generalized Pinsker inequality \cite{bolley2005weighted} used within \cite[Theorem 3.11]{eckstein2022quantitative} requires exponential moments, and is applicable to general measures.  Lemma \ref{lem:strongconvexity}  on the other hand is more specifically tailored to optimizers of DOT problems, where non-exponential moments are sufficient.


\subsubsection{Sample complexity}
\label{subsubsec:samplecomp}
While the problem of sample complexity can be regarded as a particular case of stability with respect to the marginals, the general treatment of the previous section falls short of obtaining the strongest possible results in this case. Indeed, while the previous section is only based on Lipschitz properties of the cost function, the study of sample complexity allows us to exploit differentiability properties of both divergence and cost function.

To clarify the problem, \emph{sample complexity} refers to the estimation of DOT knowing only empirical measures $\hat \mu_1^n, \hat \mu_2^n$ using $n$ sample points of the marginals. This means, we want to bound the expected difference
\[
\Delta^n := \mathbb{E}\left[\big|\OT^\varepsilon_\varphi(\mu_1, \mu_2) - \OT^\varepsilon_\varphi(\hat \mu_1^n, \hat \mu_2^n)\big|\right],
\]
where the expectation is taken over independent sample points used for $\hat \mu_1^n$ and $\hat \mu_2^n$. The main result in this regard is given in Theorem \ref{thm:samplecomplexity}, showing 
\[
\Delta^n \lesssim n^{-s/d_\mu},
\]%
where $s$ is the order of differentiability of the cost $c$ and the convex conjugate of $\varphi$, and $d_\mu > 2s$ is the intrinsic dimension of $\mu_1$ and $\mu_2$. The form of intrinsic dimension used (see Definition \ref{def:intrinsic}) is based on covering numbers and is in particular upper bounded by the entropic dimension of the support of the marginals, while also allowing for non-compact support.

To explain the idea of the proof of Theorem \ref{thm:samplecomplexity} and the relation to the literature, let us focus on compactly supported measures $\mu_1$ and $\mu_2$ on $\mathbb{R}^d$. It is known that for unregularized OT ($\varepsilon = 0$), $\Delta^n$ is of order $n^{-1/d}$ and thus a victim of the curse of dimensionality (see, e.g., \cite{fournier2015rate,2023arXiv230518636L}). On the other hand, for EOT with suitably regular cost function, $\Delta^n$ is of order $n^{-1/2}$ (see \cite{genevay2019sample,mena2019statistical}) and thus avoids the curse of dimensionality. One reason for this improved rate of convergence of EOT compared to OT is that densities of optimizers are very regular for EOT. In other words, the dual potentials for EOT are infinitely often differentiable. Thus, to bound $\Delta^n$ in the entropic case, one only has to test the difference between $(\hat \mu_1^n, \hat \mu_2^n)$ and $(\mu_1, \mu_2)$ using very regular functions. Since it is known that regularity acts as a countercurse to dimensionality (c.f.~\cite{kloeckner2018empirical}), the curse of dimensionality can thus be avoided for EOT. In this regard, we should also mention the recent work \cite{rigollet2022sample}, where similar parametric rates are obtained without using smoothness arguments but instead working with strong convexity estimates for the exponential function.

The proof of Theorem \ref{thm:samplecomplexity} follows the idea used in \cite{mena2019statistical} for EOT by establishing regularity properties of optimal dual potentials in Proposition \ref{prop:regularity}. This regularity mainly depends on two things, the regularity of the cost function $c$, and the regularity of the convex conjugate of $\varphi$, which we denote by $\psi$. Proposition \ref{prop:regularity} shows that if $c$ and $\psi$ are $s$-times continuously differentiable, then the optimal dual potentials are $(s-1)$-times continuously differentiable. The reason we lose one order is that the defining equations for the dual potentials (the analogue to the Schr\"{o}dinger equations for EOT, established in Lemma \ref{lemma:dualoptimizer}), already use the first order derivative of $c$ and $\psi$. While differentiability of order $(s-1)$ would only lead to a rate of order $n^{-(s-1)/d}$ in Theorem \ref{thm:samplecomplexity}, we additionally impose a Lipschitz condition on the order $s$ derivative of $\psi$, which transfers to the order $(s-1)$ derivative of the optimal dual potentials. 

To establish Theorem \ref{thm:samplecomplexity} using the results given by Proposition \ref{prop:regularity}, Lemma \ref{lemma:dualcon} shows that it suffices to control differences between $(\mu_1, \mu_2)$ and $(\hat \mu_1^n, \hat \mu_2^n)$ with respect to suitably regular test functions as obtained by Proposition \ref{prop:regularity}. While general results (see, e.g.,  \cite[Theorem 1.4]{kloeckner2018empirical}) would directly lead to the rate $n^{-(s-1)/d}$, we establish more specific results in Theorem \ref{thm:empiricalrate} aimed at exploiting the particular regularity structure of optimal dual potentials, as well as the intrinsic dimension of the marginals. If $C^{s-1, 1}_1$ is the class of $(s-1)$-times continuously differentiable functions where the order $(s-1)$ derivative is $1$ Lipschitz and the infinity-norms of all derivatives are bounded by $1$, and $d_{\mu} > 2s$ is the intrinsic dimension of $\mu_1$, then Theorem \ref{thm:empiricalrate} shows that
\begin{align*}
	\mathbb{E}\left[\sup_{f \in C_{1}^{s-1,1}(\mathbb{R}^d)} \int f \,d(\mu - \h \mu^n) \right] \lesssim
	\begin{cases}
	 n^{-s/d_\mu} \quad  &\text{if $d_{\mu} >2s$}, \\
	 \log(n) \, n^{-1/2} \quad &\text{if $d_{\mu}=2s$ }, \\
	 n^{-1/2} \quad  &\text{if $d_{\mu}<2s$}.
	\end{cases}
	\end{align*}

Convergence rates of empirical measures defined through testing against regular functions were also obtained in \cite{genevay2019sample,kloeckner2018empirical,mena2019statistical} by analyzing the space of test functions using reproducing kernel Hilbert spaces and Rademacher complexity, wavelet decompositions, and covering numbers, respectively. Our results additionally incorporate the information of intrinsic dimension of measures (as, e.g., used in \cite{dudley1969speed, weed2019sharp}), and our proof relies on the chaining method and Taylor expansion of test functions. The particular structure using just $(s-1)$ order derivatives with additional Lipschitz continuity for the highest order is specifically tailored to the regularity of the dual potentials as obtained by Proposition \ref{prop:regularity}. 


\subsection{Related literature}
Recent literature has caught up on establishing properties for arbitrary DOT problems aside from EOT. 
One large topic is the study of the approximation properties as $\varepsilon \rightarrow 0$.
While convergence for EOT has initially been studied in \cite{carlier2017convergence, leonard2012schrodinger} and explicit rates for discrete problems are long known (see \cite{cominetti1994asymptotic, weed2018explicit}), by now general rates of convergence (see \cite{carlier2022convergence, chizat2020faster, pal2019difference}), convergence of optimizers (see \cite{nutz2021entropic, pooladian2021entropic}) and second order expansions around $\varepsilon = 0$ (see \cite{conforti2021formula, gigli2021second}) have been established. While the initial method used in \cite{carlier2017convergence} to show convergence to the unregularized problem for $\varepsilon \rightarrow 0$ transfers to general DOT problems (see also \cite[Proposition 2.8]{eckstein2020robust}), rates of convergence have only recently been established in \cite{eckstein2022convergence}. Related convergence results for general DOT problems have also been established in \cite{lorenz2022orlicz}, where the focus is on absolutely continuous marginals.

A different string of literature focuses on computational aspects of DOT problems. Benefits of more general versions of DOT compared to EOT have been utilized for instance by \cite{blondel2018smooth,dessein2018regularized,eckstein2021computation, essid2018quadratically,lorenz2021quadratically,seguy2017large}. Recent works have also established both theoretical and computational tractability of a generalized version of Sinkhorn's algorithm (see \cite{di2020optimal, terjek2022optimal}), while \cite{lorenz2021quadratically} suggests the use of different algorithms to solve the DOT problem with quadratic divergence. 

In addition, recent works have investigated asymptotic statistical property of regularized optimal transport. 
For EOT problems, limits of dual potentials and primal optimal couplings can be found in \cite{2022arXiv220708683G,2022arXiv220707427G}, and sample complexity and central limit theorem have been established in \cite{mena2019statistical,del2023improved,rigollet2022sample} for subgaussian measures which extends the result of \cite{genevay2019sample} on bounded domains. A unified framework is provided in 
\cite{2022arXiv220504283G} to obtain limit distributions for regularized optimal transport problems including sliced Wasserstein distance,  smooth distances with compactly supported kernels, and EOT. Moreover, \cite{2022arXiv221111184S} provides limit distributions of $f$-divergences such as $KL$-divergence, $\chi^2$-divergence, squared Hellinger distance, and total variation distance, which could be potentially useful to find limit distributions of the corresponding DOT problems.

\subsection{Structure of the paper}
The remainder of the paper is structured as follows: Section \ref{sec:notation} details the notation and states preliminary results for DOT problems, and defines necessary concepts that will be used in the later sections. 
Section \ref{sec:stability} gives the results related to the quantitative stability of the DOT problem. Section \ref{sec:empiricalconvergence} gives results for the rate of convergence of empirical measures, exploiting regularity of the test functions and the intrinsic dimension of the measure. Finally, based on the results in Section \ref{sec:empiricalconvergence}, Section \ref{sec:samplecomplexity} establishes the sample complexity when approximating DOT using empirical measures of the marginals. 

\section{Notation and preliminaries}\label{sec:notation}
We first detail all used notation in Section \ref{subsec:notation}, before giving preliminary results in Section \ref{subsec:prelim}. Section \ref{subsec:prelim} recalls existence of primal and dual optimizers in Lemmas \ref{lem:primalexistence} and \ref{lemma:dualoptimizer}, and gives necessary definitions which will be used in Section \ref{sec:stability}. 
\subsection{Notation}\label{subsec:notation}
For a polish space $(Y,d_{Y})$, we denote by $\mathcal{P}(Y)$ the set of Borel probability measures on $Y$, by $\mathcal{P}_p(Y)$ for $p \in [1, \infty)$ the subset of measures $\mu$ with finite $p$-th moment, i.e., which satisfy $\int d_{Y}(y,\hat{y})^p \,\mu(dy) < \infty$ for some $\hat{y} \in Y$. For $p \in [1, \infty)$, the $p$-Wasserstein distance $W_p(\mu, \nu)$ between $\mu,\nu \in \mathcal{P}_p(Y)$ is defined via
\begin{align*}
W_p(\mu, \nu)^p &= \inf_{\pi \in \Pi(\mu, \nu)} \int d_{Y}(x, y)^p \,\pi(dx, dy).
\end{align*}

If $Y = \mathbb{R}^d$ and $s \in \mathbb{N}$, we denote by $C^s(Y)$ the set of functions $f : Y \rightarrow \mathbb{R}$ which are $s$ times continuously differentiable, where we use the notation $D^\alpha f = \frac{\partial f}{\partial_1^{\alpha_1} \partial_2^{\alpha_2} \dots \partial_d^{\alpha_d}}$ for $\alpha \in \mathbb{N}_0^d$ with $|\alpha| := \sum_{i=1}^d \alpha_i \leq s$. By $C^{s,1}(Y) \subseteq C^s(Y)$ we denote the subset of functions $f$ where $D^{\alpha} f$ is Lipschitz for $|\alpha|=s$, and by $C^{s,1}_{K}(Y) \subseteq C^{s,1}(Y)$ the subset of functions where $\|D^{\alpha} f \|_{\infty}  \leq K$ for $|\alpha| \leq s$ and $D^{\alpha} f$ is $K$-Lipschitz for $|\alpha|=s$. For a Taylor formula in several variables, we will follow \cite{folland2005higher}, and in particular use the notation $x^\alpha = x_1^{\alpha_1} \dots x_d^{\alpha_d}$ for $x \in \mathbb{R}^d, \alpha \in \mathbb{N}_0^d$. 

For $\varphi : \mathbb{R}_+ \rightarrow \mathbb{R}$ a strictly convex, lower bounded function  with $\varphi(1) = 0$ and $\lim_{x \rightarrow \infty} \varphi(x)/x = \infty$, we will always denote by 
\begin{align}\label{eq:psidef}\psi(y) := \sup_{x \geq 0}(xy - \varphi(x)) \text{ for } y \in \mathbb{R}\end{align} 
its convex dual pair, and refer to $(\varphi, \psi)$ as convex conjugates. The \emph{$\varphi$-divergence} $D_\varphi(Q,P)$ between probabilities $Q,P$ on a common space is defined by
\[
D_\varphi(Q,P) := \int \varphi\left(\frac{dQ}{dP}\right) \,dP \quad\mbox{for}\quad Q\ll P
\]
and $D_\varphi(Q, P) :=\infty$ for $Q \not\ll P$. The $D_{\varphi}$-regularized transport problem $\OT_{\varphi}$ is defined by 
\begin{align*}
\mathcal{F}_{\varphi}^{\bs{\mu}}(\pi) &:= \int c \,d\pi +  D_\varphi(\pi, P), \quad P:= \mu_{1}\otimes \cdots\otimes\mu_{N},\\
\OT_{\varphi}(\bs{\mu}) &:= \inf_{\pi \in \Pi(\bs{\mu})} \mathcal{F}_{\varphi}^{\bs{\mu}}(\pi),
\end{align*}
where $\bs{\mu}=(\mu_1,\dotso,\mu_N)$, $\mu_i \in \mathcal{P}(X_i)$ for polish spaces $(X_i, d_{X_i})$, $X = \prod_{i=1}^N X_i$, $\Pi(\bs{\mu})$ is the set of probability measures on $X$ with $i$-th marginal $\mu_i$, and $c : X \rightarrow \mathbb{R}$ is a measurable cost function. We also denote $X^{-i}=\prod_{j=1,j\not = i}^N X_i$ and $P^{-i}=\otimes_{j=1,j \not =i}^N \mu_i$ for $i=1,\dotso,N$. While usually, the problem includes an additional parameter, $\varepsilon$, as a multiplier for $D_\varphi$, we disregard $\varepsilon$ here, noting that scaling of the cost function $c$ by a factor $\frac{1}{\varepsilon}$ leads to the equivalent problem with $\varepsilon=1$.\footnote{Notably, for homogeneous cost functions like the quadratic cost, this can equivalently be considered as a rescaling of the marginal distributions, cf.~\cite{del2023improved,mena2019statistical}. In the case of bounded marginals, this gives a natural relation between the parameter $\varepsilon$ and the size of the support of the marginals, since we can always normalize to $\varepsilon = 1$ and support on the unit cube by re-scaling the cost function.} We will always assume that for some $p \in [1, \infty]$, $\mu_i \in \mathcal{P}_p(X_i)$, and if $p< \infty$, then $c$ has growth of order $p$ (i.e., $|c(x)| \leq C(1+d_X(x_0, x)^p)$ for some $x_0 \in X, C > 0$), or if $p=\infty$, then $c$ is bounded. 
For the definition of the $p$-Wasserstein distance on $X$, we will use the particular product metric $d_{X, p}(x, y):= \big(\sum_{i=1}^N d_{X_i}(x_i, y_i)^p\big)^{1/p}$.
The distance between two tuples of marginals will often be quantified by
\begin{align*}
W_{p}(\bs \mu;\tilde{\bs \mu}) :=\big(\sum_{i=1}^N W_p(\mu_i, \tilde{\mu}_i)^p\big)^{1/p}.
\end{align*} 
We denote by $\|\cdot\|_{TV}$ the total variation norm between signed measures (normalized to one for probability measures), and for two probability measures $\pi, \tilde\pi$ and a non-negative measurable function $\rho$ on the same space, we use the shorthand $\rho(\pi-\tilde{\pi})$ for the signed measure $\rho \pi - \rho \tilde{\pi}$, where $\rho\pi$ is the measure with Radon-Nikodym density $\rho$ with respect to $\pi$.  
\subsection{Preliminaries}\label{subsec:prelim}
Throughout the paper, unless otherwise specified, let $(\varphi, \psi)$ be convex conjugates, $p \in [1, \infty]$, $\mu_i \in \mathcal{P}_p(X_i)$, $i=1, \dots, N$ and assume $c: X \rightarrow \mathbb{R}$ is continuous with growth of order $p$. Note that this implies that $\int |c| \,dP < \infty$ and thus $\OT_{\varphi}(\bs \mu)$ takes a finite value.

We first recall primal attainment and duality for $\OT_{\varphi}(\bs\mu)$.
\begin{lemma}[Existence of primal optimizers and duality]\label{lem:primalexistence}
	There exists a unique $\pi^* \in \Pi(\bs \mu)$ such that
	\[
	\OT_{\varphi}(\bs \mu) = \mathcal{F}_{\varphi}^{\bs \mu}(\pi^*).
	\]
	Assuming boundedness of $c$, we have the dual formulation 
	\[
	\OT_{\varphi}(\bs \mu) = \sup_{h = \oplus_{i=1}^N h_i} \int h - \psi(h-c) \,dP,
	\]
	where the supremum is taken over measurable and bounded functions $h : X \rightarrow \mathbb{R}$ of the form $h(x) = \sum_{i=1}^N h_i(x_i)$.
	\begin{proof}
		Primal attainment is a direct consequence of the compactness of $\Pi(\bs \mu)$, continuity of $c$ and weak lower-semicontinuity of $\pi \mapsto D_\varphi(\pi, P)$. Uniqueness follows by strict convexity of $\varphi$ and thus of $D_{\varphi}$.
		
		Regarding duality we refer to \cite[Theorem 2.2]{eckstein2021computation}, where we note that therein, $h_i$ are restricted to be continuous, but the weak duality 
		\[
		\OT_{\varphi}(\bs\mu) \geq \sup_{h = \oplus_{i=1}^N h_i} \int h - \psi(h-c) \,dP,
		\]
		follows  easily by using $\psi(h-c) \geq \frac{d\pi^*}{dP} (h-c) - \varphi(\frac{d\pi^*}{dP})$.
	\end{proof}
\end{lemma}

Similar to the Schr\"{o}dinger equations for entropic OT, we will require dual attainment and corresponding first order optimality conditions for $\OT_\varphi$. In general, $\psi$ is non-decreasing and bounded from below as the convex conjugate of $\varphi$. The following condition is both a regularity condition, and also a stronger version than superlinearity at infinity for $\psi$, i.e., we don't just need that $\lim_{x \rightarrow \infty}\frac{\psi(x)}{x} = \infty$, but we require that $\psi(x)/x \gtrsim x$ for large enough $x$.
\begin{definition}[Dual regularity of $\varphi$]\label{assume1}
	We say $\varphi$ is dual regular, if $\psi \in C^1$ and there exists $C > 0$ such that ${\psi}'(x) \geq x$ for $x \geq C$ and $\psi$ is strictly convex beyond a positive point where it is 1, i.e., 
	\begin{align}\label{eq:strictconvexconjugate}
	{\exists} x_0, \,  \delta \in \mathbb{R}_+: \, \psi'(x_0) = 1 \text{ and } \psi \text{ is strictly convex on } [x_0-\delta, \infty).
	\end{align}%
\end{definition}%
Notably, \eqref{eq:strictconvexconjugate} is satisfied in many cases including entropic divergences and $\alpha$-divergences, c.f.~\cite[Theorem 11.13]{rockafellar2009variational}.

\begin{lemma}[Existence of dual optimizers]\label{lemma:dualoptimizer}
	Assume $\varphi$ is dual regular and $c$ is bounded. Then there exists $h^* = \oplus_{i=1}^N h_i^*$ measurable and bounded satisfying
	\[
	\sup_{h = \oplus_{i=1}^N h_i} \int h - \psi(h-c) \,dP = \int h^* - \psi(h^*-c) \,dP.
	\]
	Further, using the notation $x^{-i} = (x_j)_{j \neq i}$ and $P^{-i} = \otimes_{j\neq i} \mu_j$ for $x \in X$ and $i=1, \dots, N$, there exists a measurable and bounded optimizer $h^*$ satisfying 
	\[
	1 = \int {\psi}'(h^*(x) - c(x)) \,P^{-i}(dx^{-i})
	\]
	for all $x_i \in X_i$ and $i=1, \dots, N$.
	\begin{proof}
		Regarding existence, first note that for any feasible dual function $h$ and $i \in \{1, \dots, n\}$, there exists a function $g_i: X_i \rightarrow \mathbb{R}$ such that $\int {\psi}'(h-c(x_i, \cdot)-g_i(x_i)) \,dP^{-i} = 1$ for all $x_i \in X_i$ which follows since $\psi$ is bounded from below, non-decreasing and superlinear at infinity. Thus, the functional $F(x_i, v) := \int {\psi}'(h-c(x_i, \cdot)-v) \,dP^{-i}$ goes to zero for $v \rightarrow -\infty$ and to infinity for $v \rightarrow \infty$, and due to \eqref{eq:strictconvexconjugate}, it is strictly increasing in $v$ around $F(x_i, v) = 1$. Thus $g_i(x_i)$ is actually the unique solution to $F(x_i, g_i(x_i))=1$, and due to measurability of $F$ in the first argument and $\{x_i : g_i(x_i) \geq a\} = \{x_i : 1 \geq F(x_i, a)\}$ for all $a \in \mathbb{R}$, we find that $g_i$ is measurable. 
		Defining $\tilde{h}_i := h_i - g_i$ and $\var{\tilde{h}_i} := \sup_{x_i} \tilde{h}_i(x_i) - \inf_{x_i} \tilde{h}_i(x_i)$ (and similarly for $c$), we find for $x_i, x_i' \in X_i$ that
		\begin{align*}
		\int \psi'(h^{-i} + \tilde{h}_i(x_i')- c(x_i', \cdot)) \,dP^{-i}  = 1 &= \int \psi'(h^{-i}+ \tilde{h}_i(x_i) - c(x_i, \cdot) ) \,dP^{-i}\\
		&\geq  \int \psi'(h^{-i} +\tilde{h}_i(x_i)- c(x_i', \cdot)  - \var{c}) \,dP^{-i}
		\end{align*}
		and thus $\tilde{h}_i(x_i') \geq \tilde{h}_i(x_i) - \var{c}$ and hence $\var{\tilde{h}_i} \leq \var{c}$. 
		
		Further, $\tilde{h} := h - g_i$ has better dual objective than $h$. Indeed, by convexity of $\psi$, we have $\psi(h-c) \geq \psi(h-c-g_i) + \psi'(h-c-g_i)g_i$ and thus
		\begin{align*}
		\int h-g_i - \psi(h-c-g_i) \,dP &\geq \int h-g_i-\psi(h-c) \,dP + \int g_i \left(\int \psi'(h-c-g_i)\,dP^{-i} \right) \,d\mu_i\\ &= \int h-\psi(h-c) \,dP.
		\end{align*}
		%
		Now, we take a maximizing sequence $(h^n)_n$ and by the above, w.l.o.g.~$\var{h^n_i} \leq \var{c}$ for all $i$. This implies $\var{h^n} \leq N \var{c}$ and since $h^n$ is an optimizing sequence also without loss of generality $-1-\|c\|_\infty \leq \int h \,dP \leq \|c\|_\infty$ and thus $\|h^n\|_\infty \leq 2(N+1) (\|c\|_\infty+1)$.
		
		The sequence $(h^n-c)_{n}$ is uniformly bounded in $L^1(P)$ and thus by Koml\'os Lemma the sequence $(\tilde{h}^n - c) := \frac{1}{n}\sum_{i=1}^n h^i - c$ converges along a subsequence almost-surely to some $h^*-c$, where we note that $\tilde{h}^n$ is by convexity of $\psi$ still a maximizing sequence, and thus $h^*$ is a maximizer by boundedness and dominated convergence. Hereby, $h^*$ can be verified to almost surely be of the form $h^* = \oplus_{i=1}^N h_i^*$ (the property of being a direct sum transfers in almost-sure convergence) and further the above boundedness of $h^n$ transfers to $h^*$.

		Regarding first order conditions, fix $i \in \{1, \dots, N\}$. By taking continuous and bounded test functions $u_i$ and considering optimality of the map
		\[
		t \mapsto \int h^*+t u_i - \psi(h^*+tu_i - c) \, dP,
		\]
		which, after differentiating under the integral sign (which is justified, e.g., by \cite[Theorem 2.27]{folland1999real}), leads to the first order condition
		\[
		\int u_i - {\psi}'(h^*-c) u_i \,dP = 0
		\]
		and since $u_i$ is arbitrary thus 
		\[
		1 = \int {\psi}'(h^*-c) \,dP^{-i}
		\]
		$\mu_i$-almost surely. Further, since this will neither change the objective value nor the other equations for $j \neq i$, we can simply define $h^*_i$ to satisfy this equation pointwise, which is possible by the same argument as at the start of the proof, which yields the claim.
%
	\end{proof}
\end{lemma}

For the quantitative stability results for $\OT_{\varphi}$, we will require the concept of a shadow, c.f.~\cite[Definition 3.1 and Lemma 3.2]{eckstein2022quantitative}.\footnote{We refer to it as "the" shadow, even though the given construction is not always unique. Whenever we refer to the shadow of a measure, any shadow will do.} The shadow $\tilde{\pi}$ of a coupling $\pi \in \Pi(\mu_1, \dots, \mu_N)$ is a particular $W_p$-projection of $\pi$ onto the set $\Pi(\tilde{\mu}_1, \dots, \tilde{\mu}_N)$. In fact, the shadow $\tilde{\pi}$ and $\pi$ are as close as possible as two elements of the differing sets $\Pi(\mu_1, \dots, \mu_N)$ and $\Pi(\tilde{\mu}_1, \dots, \tilde{\mu}_N)$ can be. Aside from this closeness in $W_p$, the shadow $\tilde{\pi}$ enjoys a control on its divergence through $\pi$, more precisely, $D_f(\tilde{\pi}, \tilde{\mu}_1 \otimes \dots \otimes \tilde{\mu}_N) \leq D_f(\pi, \mu_1 \otimes \dots \otimes \mu_N)$. Combining both properties makes the shadow an attractive tool to study stability aspects of regularized optimal transport. The formal definition follows.
\begin{definition}[Shadow]\label{def:shadow}
	Let $\mu_i, \tilde{\mu}_i \in \mathcal{P}_p(X_i)$, $i=1,\dotso,N$. For $\pi \in \Pi(\mu_1, \dots, \mu_N)$, we define its \emph{shadow} $\tilde{\pi} \in \Pi(\tilde{\mu}_1, \dots, \tilde{\mu}_N)$ as the second marginal of 
	\[
	\pi \otimes K \in \mathcal{P}(X \times X),
	\]
	where $K$ is the stochastic kernel given as the point-wise product measure
	\[
	K(x) = K_1(x_1) \otimes K_2(x_2) \otimes \dots \otimes K_N(x_N),
	\]
	where $\theta_i = \mu_i \otimes K_i \in \Pi(\mu_i, \tilde{\mu}_i)$ are $W_p(\mu_i, \tilde{\mu}_i)$ optimizers.
\end{definition}

We also use a condition on cost functions $c$ which generalizes Lipschitz continuity, and holds for instance also for $c(x) = \|x_2-x_1\|^p$ for $x \in \mathbb{R}^d \times \mathbb{R}^d$, c.f.~\cite[Lemma 3.5, Example 3.6]{eckstein2022quantitative}.
\begin{definition}[Weakened Lipschitz condition of $c$]\label{def:lipschitz}
	For $L > 0$, we say that $c$ satisfies $(A_L)$ for marginals $ \bs \mu, \tilde{\bs \mu} \in \mathcal{P}_p(X)$, if 
	\begin{align}\label{eq:AL}
	\left|\int c \,d(\pi - \tilde{\pi})\right| \leq L W_p(\pi, \tilde{\pi})\tag{$A_L$}
	\end{align}
	for all $\pi \in \Pi(\bs \mu), \tilde{\pi} \in \Pi(\tilde{\bs \mu})$.
\end{definition}

Using the concept of the shadow, the stability of the value of $\OT_{\varphi}$ was shown in \cite[Theorem 3.7]{eckstein2022quantitative}.
\begin{lemma}[Continuity of $\OT_{\varphi}$]\label{lem:continuityofOT}
	Assume $c$ satisfies \eqref{eq:AL} for marginals $\bs \mu, \tilde{\bs \mu} \in \mathcal{P}_p(X)$. Then
	\[
	\left| \OT_{\varphi}(\bs \mu) - \OT_{\varphi}(\tilde{\bs \mu})\right| \leq L W_p(\bs \mu, \tilde{\bs \mu}).
	\]
\end{lemma}

\section{Quantitative stability of regularized optimal transport}\label{sec:stability}

%
%

For the quantitative stability result for the optimizers of $\OT_{\varphi}(\bs\mu)$, we will rely on the respective quantitative stability of the optimal values given in Lemma \ref{lem:continuityofOT}.

Recall that a strongly convex function $\varphi \in C^2(\mathbb{R})$ with parameter $m > 0$ satisfies $\varphi''(x) \geq m$, or equivalently $\frac{1}{\varphi''(x)} \leq \frac{1}{m}$. This assumption is too strong for our purposes, since it is not satisfied, e.g., for $\varphi(x) = x \log x$. For the purposes of our paper, a weaker strong convexity assumption is sufficient, which is given in the following definition.
\begin{definition}[Strong convexity assumption for $\varphi$]\label{def:strongconvex}
	For $\lambda_1, \lambda_2 \geq 0$, we say $\varphi$ is $(\lambda_1, \lambda_2)$-convex, if $\varphi \in C^2(\mathbb{R}_+)$ and
	\[
	\frac{1}{\varphi''(x)} \leq \lambda_1 + \lambda_2 x 
	\]
	for all $x \in (0, \infty)$.
\end{definition}

\begin{example}
	We showcase relevant instances of divergence functions $\varphi$ which are both dual regular (c.f.~Definition \ref{assume1}) and also $(\lambda_1, \lambda_2)$-convex for some $\lambda_1, \lambda_2 \geq 0$ (c.f.~Definition \ref{def:strongconvex}).
	
	First, clearly, the entropic case is covered, i.e., $\varphi_{\rm ent}(x) := x \log x$ is both dual regular and $(0,1)$-convex.
	
	Second, divergences of power form satisfy the constraints as long as the power is equal to or below 2. We consider $\alpha$-divergences, which clearly includes the popular $\chi^2$-divergence for $\alpha=2$. Let $\alpha \in (1,2]$ and $\varphi_{\alpha}(x):= \frac{x^{\alpha}-\alpha(x-1)-1}{\alpha(\alpha-1)}$. That $\varphi_{\alpha}$ is dual regular can be obtained from the fact that the convex conjugate $\psi(y)$ equals, up to scaling and shifting, the function $y_+^\beta$, where $\frac{1}{\alpha} + \frac{1}{\beta} = 1$ and thus $\beta \geq 2$. Hence, both $\psi \in C^1$, and further $\psi(x)/x \gtrsim x$ for large $x$ follows. Regarding $(\lambda_1, \lambda_2)$-convexity, $\varphi_\alpha$ has second order derivative $\varphi_{\alpha}''(x)=x^{\alpha-2}$. Therefore by Young's inequality $$\frac{1}{\varphi_{\alpha}''(x)}=x^{2-\alpha} \leq (\alpha-1)+(2-\alpha)x,$$
	and thus $\varphi_{\alpha}$ is $(\alpha-1,2-\alpha)$-convex.
\end{example}

\begin{lemma}[Strong convexity of $\mathcal{F}_{\varphi}^{\bs \mu}$ at optimizer]\label{lem:strongconvexity}
	Assume $\varphi$ is $(\lambda_1, \lambda_2)$-convex, $\pi^*$ is an optimizer of $\OT_{\varphi}(\bs \mu)$, and $\pi \in \Pi(\bs \mu)$. Then, for any measurable $\rho : X \rightarrow \mathbb{R}_+$, 
	\[
	\|\rho (\pi^* - \pi)\|_{TV}^2 \leq C \left(\int |\rho |^2 \,d(P+\pi^* + \pi) \right)  \left(\mathcal{F}_{\varphi}^{\bs \mu}(\pi) - \mathcal{F}_{\varphi}^{\bs \mu}(\pi^*)\right),
	\]
	where $C = 4\max\{\lambda_1,\lambda_2\}$.
\end{lemma}
\begin{proof}
	Without loss of generality, we assume that $\cF^{\bs \mu}_\varphi(\pi) < +\infty$. Since $\pi^*$ is an optimizer, the function $[0,1] \ni t \mapsto  \cF^{\bs \mu}_\varphi ((1-t) \pi^*+ t \pi  )$ is  increasing in $t$, and hence its first order derivative is non-negative at $t=0$. Therefore we obtain that
	\begin{align}\label{eq:strongconv1}
	\int c \, d (\pi -\pi^*) + \int \varphi'(\rho_{\pi^*})( \rho_{\pi}-\rho_{\pi^*}) \, dP \geq 0,
	\end{align}
	where $ \rho_{\pi}:= \frac{d \pi }{d P}$, $ \rho_{\pi^*}:= \frac{d \pi^*}{dP}$.  
	
	By the Taylor expansion of function $\varphi$ 
	\begin{align*}
	\varphi(y) =\varphi(x)+(y-x) \varphi'(x)+ (y-x)^2 \int_0^1 dt \int_0^t \varphi''(x+s(y-x) ) ds,
	\end{align*}	
	we get that 
	\begin{align*}
	\cF^{\bs \mu}_\varphi (\pi)-\cF^{\bs \mu}_\varphi( \pi^*) = &\int c \, d (\pi -\pi^*) + \int \varphi(\rho_{\pi}) - \varphi(\rho_{\pi^*}) \, d P \\
	= & \int c \, d (\pi -\pi^*) + \int \varphi'(\rho_{\pi^*})( \rho_{\pi}-\rho_{\pi^*}) \, d P \\
	&+  \int  (\rho_{\pi}-\rho_{\pi^*})^2 \int_0^1 dt \int_0^t \varphi''(\rho_{\pi^*}+ s(\rho_{\pi}-\rho_{\pi*})) \, ds \, dP \\
	\geq & \int  \int_0^1 (\rho_{\pi}-\rho_{\pi^*})^2 (1-t) \varphi''(\rho_{\pi^*}+ t(\rho_{\pi}-\rho_{\pi*})) \, dt \,dP,
	\end{align*}	
	where the last inequality is due to \eqref{eq:strongconv1}.

	Denoting $\Delta = \rho_{\pi}-\rho_{\pi^*}$, 
	by the Cauchy-Schwartz inequality, we obtain 
	\begin{align*}
	& \left( \int  \int_0^1 \Delta^2 (1-t) \varphi''(\rho_{\pi^*}+ t \Delta) \, dt \,d P \right) \left( \int \int_0^1 \frac{(1-t) \rho^2}{\varphi''(\rho_{\pi^*}+t\Delta)} \, dt \, d P  \right) \notag \\
	& \geq \left( \int \int_0^1 (1-t)  |\rho \Delta|  \, dt \, d P \right)^2 =\frac{1}{4} \lVert \rho(\pi-\pi^*) \rVert_{TV}^2.
	\end{align*}
	According to the $(\lambda_1,\lambda_2)$-convexity of $\varphi$, the second term on the left hand side is bounded from above by 
	\begin{align*}
	\max\{ \lambda_1,\lambda_2\} \int |\rho(x)|^2 \, d (P+\pi^*+\pi).
	\end{align*}
	Therefore we can bound the weighted total variation of $\pi,\pi^*$ from above by 
	\begin{align*}
	\frac{1}{4} \lVert \rho ( \pi- \pi^*) \rVert_{TV}^2 \leq  \max\{ \lambda_1,\lambda_2\}  \left( \mathcal{F}_\varphi^{\bs \mu} (\pi)-\mathcal{F}_\varphi^{\bs \mu}(\pi^*) \right) \int |\rho(x)|^2 \, d (P+\pi^*+\pi),
	\end{align*}
	which completes the proof.
\end{proof}

The following theorem provides a quantitative stability of optimizers for a large class of divergence regularizations. Furthermore, in the case of entropic regularization, it recovers  \cite[Theorem 3.11]{eckstein2022quantitative} under weaker moment assumptions. As explained in Section \ref{subsubsec:quantstab}, the reason for this improved moment assumption is that the strong convexity estimate in \cite[Theorem 3.11]{eckstein2022quantitative} makes use of the generalized Pinsker inequality from \cite{bolley2005weighted}, which is applicable to general probability measures, but requires exponential moments. The following result instead uses the strong convexity estimate of Lemma \ref{lem:strongconvexity}, which is more specifically tailored to optimizers of $\OT_{\varphi}(\bs\mu)$.
\begin{theorem}[Quantitative stability of optimizers of $\OT_{\varphi}$]\label{thm:quantstability}
	Assume $\varphi$ is $(\lambda_1, \lambda_2)$-convex and $c$ satisfies \eqref{eq:AL} for marginals $\bs \mu, \tilde{\bs \mu} \in \mathcal{P}_p(X)$. Let $\pi^*, \tilde{\pi}^*$ be optimizers of $\OT_{\varphi}(\bs \mu)$, $\OT_{\varphi}(\tilde{\bs \mu})$, respectively, let $q \in [1, p]$ and
	$
	\Delta := W_p(\bs \mu, \tilde{\bs \mu}),
	$
	and assume $\mu \in \mathcal{P}_{2q}$.
	Then
	\begin{align*}
	W_q(\pi^*, \tilde{\pi}^*) &\leq N^{(\frac{1}{q}-\frac{1}{p})} \Delta + C \big(L\Delta\big)^{\frac{1}{2q}},
	\end{align*}
	where $C$ is a constant only depending on $\lambda_1,\lambda_2$ and the $2q$-th moment of $\bs \mu$.
	\begin{proof}
		
		Take $\pi \in \Pi(\bs \mu)$ to be a shadow of $\tilde{\pi}^*$. From \cite[Lemma 3.2]{eckstein2022quantitative}, we find
		\begin{align*}
		W_p(\tilde{\pi}^*, \pi)=W_p( \bs \mu, \tilde{\bs \mu}), \quad \cF^{\bs \mu}_\varphi (\pi) \leq \OT_{\varphi}(\tilde{\bs \mu}) + L \Delta. 
		\end{align*}
		Thus, by Lemma \ref{lem:continuityofOT},
		\[
		\cF^{\bs \mu}_\varphi (\pi) - \cF^{\bs \mu}_\varphi (\pi^*) \leq 2L\Delta.
		\]
		Take $\rho(x) = d_X(x_0, x)^q$ in Lemma \ref{lem:strongconvexity}, and use the following inequality from  \cite[Proposition 7.10]{villani2021topics}
		\[
		W_q(\pi^*, \pi)^q \leq 2^{q-1} \| \rho (\pi^* - \pi)\|_{TV},
		\]
		which yields that 
		\begin{align*}
		W_q(\pi^*,\pi)^{2q} \leq \tilde{C}   \left(\mathcal{F}_{\varphi}^{\bs \mu}(\pi) - \mathcal{F}_{\varphi}^{\bs \mu}(\pi^*)\right) \leq \tilde{C} 2 L \Delta,
		\end{align*}
		where $\tilde{C}$ is, in view of Lemma \ref{lem:strongconvexity}, a constant only depending on $q, \lambda_1,\lambda_2$ and the $2q$-th moment of $\bs \mu$. Now, by the triangle inequality, we conclude 
		\begin{align*}
		W_q( \pi^*, \tilde{\pi}^*) \leq W_q(\tilde{\pi}^*, \pi)+W_q(\pi^*, \pi) \leq N^{(\frac{1}{q}-\frac{1}{p})} \Delta+C\big(L\Delta\big)^{\frac{1}{2q}},
		\end{align*}
		where $C = (\tilde{C} 2)^{1/2q}$ only depends on $\lambda_1,\lambda_2$ and the $2q$-th moment of $\bs \mu$.
	\end{proof}
\end{theorem}

%

\section{Intrinsic dimension and regularity of test functions}\label{sec:empiricalconvergence}


We have seen in Section \ref{sec:stability}, respectively Lemma \ref{lem:continuityofOT}, that quantitative stability estimates for DOT can be obtained using Lipschitz properties of the cost functions. However, for sample complexity, the literature on EOT \cite{genevay2019sample,mena2019statistical} has established that additional regularity properties can be exploited, and to an extent the same is true even for unregularized OT \cite{manole2021sharp}. This section provides the specific tools for convergence rates of empirical measures with respect to test functions arising from dual optimal solutions for general DOT problems, allowing us to exploit regularity in the sample complexity of DOT as well in Section \ref{sec:samplecomplexity}. The main result, Theorem \ref{thm:empiricalrate}, may be of interest on its own, as to the best of the authors' knowledge, it is the first result measuring empirical rates of convergence, both with higher order differentiability of test functions, and a concept of intrinsic dimension (see Definition \ref{def:intrinsic}).

We fix $\mu \in \mathcal{P}(\mathbb{R}^d)$ and for i.i.d.~random variables $Z_1, Z_2, \dots$ with $Z_1 \sim \mu$, we denote by $ \h{\mu}^n := \frac{1}{n} \sum_{i=1}^n \delta_{Z_i}$ for $n \in \mathbb{N}$ the (random) empirical measures of $\mu$. 

The following introduces the concept of intrinsic dimension used in this paper. Roughly speaking, the measure $\mu$ has intrinsic dimension $d_\mu$, if most of its support can be covered as efficiently as a $d_\mu$-dimensional unit cube. The precise meaning of "most" in this context depends on the regularity $s$ of the test functions that we use. The general concept is adapted from \cite{dudley1969speed}, and we discuss the relation in more detail in Remark \ref{rem:discussionintrinsic} below.

%
%
%
%

\begin{definition}[Intrinsic dimension]\label{def:intrinsic}
	Let $\mu \in \mathcal{P}(\mathbb{R}^d)$, $s \in \mathbb{N}$, $d_\mu \in (0, \infty)$. We say that $\mu$ satisfies $I(d_\mu, s)$, if there exists $K > 0$ such that for every $\varepsilon > 0$, there exists $\Omega_\varepsilon \subset \mathbb{R}^d$ satisfying 
	$\mu(\Omega_\varepsilon^c) \leq \mathbbm{1}_{\{d_{\mu}>2s\}}\varepsilon^{sd_\mu/(d_\mu-2s)}$
	and a partition $\mathcal{A}_\varepsilon$ of $\Omega_\varepsilon$ s.t.
	\begin{align*}
	D_{\mathcal{A}_\varepsilon} := \sup_{A \in \mathcal{A}_\varepsilon} \sup_{x, y \in A} \|x-y\| \leq \varepsilon ~ \text{and} ~
	|\mathcal{A}_\varepsilon| \leq K \varepsilon^{-d_\mu},
	\end{align*}
	where $|\mathcal{A}_\varepsilon|$ denotes the number of sets in the partition $\mathcal{A}_{\varepsilon}$. 
\end{definition}

\begin{remark}\label{rmk:dim}
Note that our definition of intrinsic dimension separates into two cases. We require $\mu(\Omega_\varepsilon^c) \leq \varepsilon^{sd_\mu/(d_\mu-2s)}$ if $d_{\mu} >2s$, and $\mu(\Omega_\varepsilon^c) =0$ if $d_{\mu} \leq 2s$. It can be seen that condition $I(d_{\mu},s)$ is stronger than $I(d_{\mu},s')$ for any $s>s'$. So if a measure $\mu$ satisfies $I(d_{\mu},s)$, it satisfies $I(d_{\mu},s')$ for all $s' <s$.
\end{remark}

\begin{remark}\label{rem:discussionintrinsic}
The notion of intrinsic dimension of Definition \ref{def:intrinsic} is closely connected to different versions of dimension in the literature. First, the definition is inspired by \cite{dudley1969speed}, and the notion of intrinsic dimension therein corresponds to the one we use for $s = 1, d_\mu > 2$. Interestingly, basically the same concept of dimension as in Definition \ref{def:intrinsic} was already used by \cite[$d_p^*$ in Definition 4]{weed2019sharp}. Among others, it is shown in \cite{weed2019sharp} that if the support of $\mu$ has Minkowski-dimension $d_M$ (also called entropic or box dimension), then $\mu$ satisfies $I(d_\mu, s)$. Further relations to Hausdorff and quantization dimension are given in \cite[Chapter 4]{weed2019sharp} and \cite[Chapter 11]{graf2007foundations}.

While the goal in \cite{weed2019sharp} was to generalize the work by \cite{dudley1969speed} from $W_1$ to $W_p$, our goal is to generalize Lipschitz test functions to higher order regularity. It is not surprising that both questions use the same generalization of the notion of dimension compared to the work in \cite{dudley1969speed}. Indeed, a larger $s > 0$ means a faster decay of the "tail-part" $\Omega_\varepsilon^C$ in Definition \ref{def:intrinsic}. In our work, the "non-tails" (i.e., $\Omega_\varepsilon$) converge with a faster speed since test functions are more regular. This means the assumption of larger $s$ is used for the "tail-part" (i.e., $\Omega_\varepsilon^C$) to keep up with the faster speed of the non-tail part. Similarly, since \cite{weed2019sharp} deals with $W_p$, the "tail-part" needs to be controlled with respect to $p$-th moments and not just first moments as in \cite{dudley1969speed}, again motivating the same notion of intrinsic dimension.
\end{remark}

The following is the main result of this section.
\begin{theorem}\label{thm:empiricalrate}
	Assume $\mu$ satisfies $I(d_\mu, s)$ for some $s \in \mathbb{N}$, $d_\mu \in (0,\infty)$. Then
	\begin{align*}
	\mathbb{E}\left[\sup_{f \in C_{1}^{s-1,1}(\mathbb{R}^d)} \int f \,d(\mu - \h \mu^n) \right] \lesssim
	\begin{cases}
	 n^{-s/d_\mu} \quad  &\text{if $d_{\mu} >2s$}, \\
	 \log(n) \, n^{-1/2} \quad &\text{if $d_{\mu}=2s$ }, \\
	 n^{-1/2} \quad  &\text{if $d_{\mu}<2s$}.
	\end{cases}
	\end{align*}
\end{theorem}

\begin{remark}
Given any set of test functions $\mathcal{F}$, the so-called integral probability metric (IPM) is defined as $d_{\mathcal{F}}(\mu,\nu):= \sup \left\{\left|\int f \, d(\mu-\nu) \right|: \, f \in \mathcal{F} \right\}$, cf.~\cite{muller1997integral}. Choosing different class of test functions, various popular distances, including Wasserstein-$1$ metric, total variation distance, and Fourier-Wasserstein distance, can be obtained.  As shown in \cite{10.1214/12-EJS722}, IPM is also closely related to the problem of binary classification. For any $\mathcal{F}\subseteq C_1^{s-1,1}(\mathbb{R}^d)$, together with triangle inequality, our theorem provides a convergence rate of $d_{\mathcal{F}}(\hat{\mu}_n,\hat{\nu}_n) \to d_{\mathcal{F}}(\mu,\nu)$ in expectation.  


\end{remark}

The proof will be given at the end of the section, and we first state some preliminary results that will be used. The idea is based on the proof of \cite[Theorem 3.2]{dudley1969speed}, which is to use a refinement of partitions of a suitable form, c.f.~Lemma \ref{lem:refinement}. While \cite[Theorem 3.2]{dudley1969speed} uses Lipschitz continuity to pass from one partition to the next coarser one, we will use a Taylor expansion to this end. Lemma \ref{lem:decomposition} gives the basic result for how we deal with the sampling error. Lemmas \ref{lemma:estimateE} and \ref{lem:induction_s} will inductively control the errors when passing from one partition to the next coarser one, while the proof collects the estimate for the finest partition and suitably balances the occurring parameters. As noticed in Remark~\ref{rmk:dim}, $\mu$ satisfying $I(d_{\mu},s)$ implies the condition $I(d_{\mu},s')$ for any $s'<s$. Therefore in the case of $d_{\mu}<2s$, it is sufficient to prove the result for  $d_{\mu} \in (2s-2,2s)$.

\begin{lemma}\label{lem:decomposition} 
	Let $g: \Omega \rightarrow \mathbb{R}$ be measurable and bounded and $\mathcal{A}$ be a partition such that $\cup_{A \in \cA} A \subset \Omega$. Then
	\[
	\mathbb{E}\left[ \sum_{A \in \mathcal{A}} \left| \int_A g \,d(\mu - \h \mu^n) \right|\right] \leq \|g \|_\infty \left( \frac{|\mathcal{A}|}{n} \right)^{1/2}. 
	\]
	\begin{proof}
		Recall that $\h \mu^n =\frac{\sum_{i=1}^n \delta_{Z_i}}{n}$, and hence $\int_A g \, d \h \mu^n=\frac{\sum_{i=1}^n g_A(Z_i)}{n}$, where $g_A(x):=g(x) \mathbbm{1}_A(x)$ and $\{ g_A(Z_i)\}_{i=1,\dotso,n}$ are \emph{i.i.d.} with expectation $\int_A g \, d\mu$. Therefore we have $$\mathbb{E} \left[ \left|\int_A g \, d(\mu-\h \mu^n) \right|^2 \right]=\frac{\text{Var}(g_A(Z_1))}{n} \leq \frac{1}{n} \int_A g^2 \, d\mu. $$ Now by Cauchy-Schwartz or Jensen's inequality, we obtain that 
		\begin{align*}
		\mathbb{E}\left[ \sum_{A \in \mathcal{A}} \left| \int_A g \,d(\mu - \h \mu^n) \right|\right] & \leq \sqrt{|\cA|} \,\sqrt{ \mathbb{E}\left[\sum_{A \in \cA} \left| \int_A g \, d(\mu-\h \mu^n)\right|^2 \right]} \\
		& \leq \sqrt{ \frac{|\mathcal{A}|}{n}   \int g^2 \, d \mu } \leq \|g\|_\infty \left( \frac{|\mathcal{A}|}{n} \right)^{1/2}. 
		\end{align*}
	\end{proof}
\end{lemma}

In the following, we will show that if we choose partitions $\mathcal{A}_T, \mathcal{A}_{T-1}, \dots, \mathcal{A}_1$ for different levels of $\varepsilon$, namely, $\varepsilon_t = 3^t \varepsilon$, then the partitions may without loss of generality be taken as refinements of each other. In the remainder of this section, we always consider $\mathcal{A}_T, \dots, \mathcal{A}_1$ to be as in the following lemma. 
\begin{lemma}[Partitions can be chosen as refinements]\label{lem:refinement}
	Assume $\mu$ satisfies $I(d_\mu, s)$ for some $s \in \mathbb{N}$, $d_\mu \in (0, \infty)$. Then there exists a constant $K > 0$ such that for all $T \in \mathbb{N}$ and $\varepsilon > 0$ with $\varepsilon 3^T \leq 1$, there exists $\bar{\Omega} \subset \mathbb{R}^d$ and refinements of partitions $\mathcal{A}_T, \mathcal{A}_{T-1}, \dots, \mathcal{A}_1$ of $\bar{\Omega}$ s.t.
	\[
	|\mathcal{A}_t| \leq K (\varepsilon 3^t)^{-d_\mu},~ |D_{\mathcal{A}_t}| \leq 3^t \varepsilon \text{ and } \mu(\bar{\Omega}^c) \leq \mathbbm{1}_{\{d_{\mu}>2s\}}K (3^T \varepsilon)^{sd_\mu / (d_\mu - 2s)}.
	\]
	\begin{proof}
		Aside from the last part of the claim, this result is given in \cite[Proof of Theorem 3.2]{dudley1969speed}. We shortly recap the construction, and prove the claim for $d_{\mu}>2s$ since the case of $d_{\mu}\leq 2s$ is trivial. In what follows, the constant $K$ may change from line to line. From Definition \ref{def:intrinsic} (using $\tilde{K}$ instead of $K$ therein), we take a partition $\tilde{\mathcal{A}}_t$ of $\Omega_t$ satisfying $|D_{\tilde{\mathcal{A}_t}}| \leq 3^{t-1} \varepsilon$, $|\tilde{\mathcal{A}}_t| \leq \tilde{K} (\varepsilon 3^{t-1})^{-d_\mu} \leq K (\varepsilon 3^t)^{-d_\mu}$, and $\mu(\Omega_t^c) \leq (\varepsilon 3^{t-1})^{sd_{\mu}/(d_\mu-2s)}$. Taking $\bar{\Omega} := \cap_{t=1}^T \Omega_t$ satisfies the desired inequality with $K=\frac{1}{3^{sd_\mu/(d_\mu-2s)} - 1}$ by the geometric series, since
		\[
		\mu(\bar{\Omega}^c) \leq \sum_{t=0}^{T-1} (\varepsilon 3^t)^{sd_\mu/(d_\mu-2s)} \leq K(3^T \varepsilon)^{{sd_\mu/(d_\mu-2s)}}.
		\]
		In what follows, consider all sets to be intersected with $\bar{\Omega}$ (and if this leads to an empty set, disregard it).
		We take $\mathcal{A}_1 := \tilde{\mathcal{A}}_1$. Using the notation $\tilde{\mathcal{A}}_t = \{\tilde{A}_{t, 1}, \dots, \tilde{A}_{t, n_t}\}$, we proceed inductively to define $\mathcal{A}_t = \{A_{t, 1}, \dots, A_{t, n_t}\}$ as follows: 
		\begin{align*}
		A_{t, 1} &:= \cup\{A_{t-1, j} : A_{t-1, j} \cap \tilde{A}_{t, 1} \neq \emptyset, j=1, \dots, n_{t-1}\}, \\
		A_{t, i} &:= \cup\{A_{t-1, j} : A_{t-1, j} \cap \tilde{A}_{t, i} \neq \emptyset \text{ and } A_{t-1, j} \cap A_{t, k} = \emptyset\\
		&\hspace{2.3cm} \text{ for all } k=1, \dots, i-1, j=1, \dots, n_{t-1}\},
		\end{align*}
		for $i=1, \dots, n_{t}$, $t=2, \dots, T$. The refinement property is clear by construction, and further $D_{\mathcal{A}_1} = D_{\tilde{\mathcal{A}}_1}$ and $D_{\mathcal{A}_t} \leq 2 D_{\mathcal{A}_{t-1}} + D_{\tilde{\mathcal{A}}_t}$, which inductively yields the claim.
	\end{proof}
\end{lemma}

For the proof of the main theorem, we require the following notation.
For $\tilde{s} \in \{1, \dots, s\}$,  $g \in L^\infty$, and $i \in \{1, \dots, T\}$, define
\[
E_{\tilde{s}, g, i} := \mathbb{E}\left[\sup_{f \in C^{\tilde{s}-1,1}_1}\Big|\sum_{A_i \in \mathcal{A}_i} \int_{A_i} f(x_{A_i}) g(x) \,(\mu - \h \mu^n)(dx) \Big|\right],
\]
where $x_{A_i}$ are some fixed elements in $A_i$. We further define
\[
M_{i, g} := \mathbb{E}\left[\sum_{A_i \in \mathcal{A}_i} \left| \int_{A_i} g \,d(\mu-\h\mu^n)\right|\right].
\]
By $g_{Z}$ for $Z>0$ we denote some generic function in $L^\infty$ satisfying $\|g_Z\|_\infty \leq Z \|g \|_\infty$, and we abbreviate $D_t := D_{\mathcal{A}_t}$

\begin{lemma}\label{lemma:estimateE}
	For any ${\tilde s} \in \{1,\dotso,s\}$, $ g \in L^\infty$ and $ i \in \{1,\dotso, T-1\}$, we have that 
	\begin{align*}
	E_{ {\tilde s}, g, i}  \leq E_{{\tilde s}, g, T} + \sum_{j=i}^{T-1} \left[ \sum_{1 \leq |\alpha| \leq {\tilde s}-1} \frac{1}{\alpha!} E_{{\tilde s}-|\alpha|, g_{D_{j+1}^{|\alpha|}}, j+1} + \sum_{|\alpha|=\tilde{s}} \frac{1}{\alpha!}D_{j+1}^{\tilde s} M_{j, g}. \right]
	\end{align*}
\end{lemma}
\begin{proof}
	For any $A_i \in \cA_i$, since $\cA_i$ is a refinement of $\cA_{i+1}$ there exists a unique $A_{i+1} \in \cA_{i+1}$ such that $A_i \subset A_{i+1}$. Then, for arbitrary $f \in C^{{\tilde s}-1,1}_1$, by the Taylor expansion in multi-index notation (following, e.g., \cite{folland2005higher}) $$f(x_{A_i})= f(x_{A_{i+1}})+\sum_{1 \leq |\alpha| \leq {\tilde s}-1} \frac{D^{\alpha}f(x_{A_{i+1}})}{\alpha!}(x_{A_i}-x_{A_{i+1}})^{\alpha}+R^f(x_{A_i}, x_{A_{i+1}}),$$ 
	where 
	\[
	R^f(x_{A_i}, x_{A_{i+1}}) = \sum_{|\alpha|=\tilde{s}-1} \frac{1}{\alpha!} \Big( D^\alpha f(\xi_{A_i, A_i+1}) - D^\alpha f(x_{A_{i+1}})\Big)(x_{A_i} - x_{A_{i+1}})^\alpha
	\]
	for $\xi_{A_i, A_i+1} = x_{A_{i+1}} + c_{A_i, A_{i+1}} (x_{A_{i}} - x_{A_{i+1}})$ for some $c_{A_i, A_{i+1}} \in (0, 1)$. And thus, by 1-Lipschitz continuity of $D^\alpha f$ for $|\alpha| = {\tilde{s}}-1$, we obtain
	\[
	|R^f(x_{A_i}, x_{A_{i+1}})| \leq \sum_{|\alpha|=\tilde{s}-1} \frac{1}{\alpha!} D_{i+1}^{\tilde{s}}.
	\]
	Define functions $g^{\alpha}$ for each index $\alpha$ via $$g^{\alpha}(x):= \sum_{A_i \in \cA_i}\mathbbm{1}_{\{x \in A_i\}} g(x) (x_{A_i}-x_{A_{i+1}})^{\alpha},$$
	whose sup norm is clearly bounded by $D_{i+1}^{|\alpha|}  \|g \|_\infty$, and thus an element of $g_{D_{i+1}^{|\alpha|}}$. 
	Now plugging the Taylor expansion into the definition of $E_{{\tilde s},g,i}$, we obtain that 
	\begin{align*}
	& E_{{\tilde s},g,i}\leq \mathbb{E}\left[\sup_{f \in C^{\tilde s-1,1}_1}\left|\sum_{A_{i+1} \in \cA_{i+1}} \int_{A_{i+1}} f(x_{A_{i+1}})g(x) \, (\mu-\h \mu^n)(dx) \right|\right] \\
	&+\sum_{1 \leq |\alpha| \leq {\tilde s}-1} \frac{1}{\alpha!} \mathbb{E} \left[\sup_{f \in C^{\tilde s-1,1}_1} \left| \sum_{A_{i+1} \in \cA_{i+1}} \int_{A_{i+1}} D^{\alpha}f(x_{A_{i+1}}) g^{\alpha}(x) \, (\mu-\h \mu^n)(dx) \right|  \right]  \\
	&+ \mathbb{E}\left[\sup_{f \in C^{\tilde s-1,1}_1} \left|\sum_{A_i \in \cA_i} \int_{A_i} R^f(x_{A_i},x_{A_{i+1}})g(x) \, (\mu-\h\mu^n)(dx) \right|\right]. 
	\end{align*}
	Now the first term on the right hand side is just $E_{{\tilde s},g,i+1}$, and the last term is clearly bounded from above by $\sum_{|\alpha|=\tilde{s}} \frac{1}{\alpha!}D_{i+1}^{\tilde s} M_{i, g}$. 
	For each term in the second line corresponding an index $\alpha$, we have $D^{\alpha} f \in C^{{\tilde s}-|\alpha|-1,1}_1$ and $g^{\alpha}(x)$ is an element of $g_{D_{i+1}^{|\alpha|}}$. Therefore we get the estimate
	\begin{align*}
	E_{{\tilde s}, g, i} &\leq E_{{\tilde s}, g, i+1} + \sum_{1 \leq |\alpha| \leq {\tilde s}-1} \frac{1}{\alpha!} E_{{\tilde s}-|\alpha|, g_{D_{i+1}^{|\alpha|}}, i+1} + \sum_{|\alpha|=\tilde{s}} \frac{1}{\alpha!} D_{i+1}^{\tilde s} M_{i, g}. \\
	\end{align*}
	This inequality holds for any $i \in \{1, \dotso, T-1\}$. Summing them up from $i$ to $T-1$, we obtain our claimed result 
	\begin{align*}
	E_{ {\tilde s}, g, i}  \leq E_{{\tilde s}, g, T} + \sum_{j=i}^{T-1} \left[ \sum_{1 \leq |\alpha| \leq {\tilde s}-1} \frac{1}{\alpha!} E_{{\tilde s}-|\alpha|, g_{D_{j+1}^{|\alpha|}}, j+1} + \sum_{|\alpha|=\tilde{s}} \frac{1}{\alpha!}D_{j+1}^{\tilde s} M_{j, g} \right].
	\end{align*}
\end{proof}

\begin{lemma}
	\label{lem:induction_s}
	There exists a constant $C$ independent of $\varepsilon, n, T$ such that for $\tilde{s} \in \mathbb{N}$,  $g \in L^\infty$ and $i \in \{1, \dots, T\}$, the following bound holds:
	\[
	E_{\tilde{s}, g, i} \leq 
	\begin{cases}
	C \|g\|_\infty \left(3^{i(\tilde{s}-\frac{d_{\mu}}{2})}\frac{\varepsilon^{\tilde{s}-d_{\mu}/2}}{n^{1/2}} + ((\varepsilon 3^T)^{\tilde{s}-1}+1) \left(\frac{|\mathcal{A}_T|}{n}\right)^{1/2} \right), \quad & \text{if $d_{\mu} >2 \tilde s$}, \\
	C \|g\|_\infty \left(\frac{T}{n^{1/2}} + ((\varepsilon 3^T)^{\tilde{s}-1}+1) \left(\frac{|\mathcal{A}_T|}{n}\right)^{1/2} \right), \quad & \text{if $d_{\mu} =2 \tilde s$}, \\
	C \|g\|_\infty \left(\frac{(3^T\varepsilon)^{\tilde{s}-d_{\mu}/2}}{n^{1/2}} + ((\varepsilon 3^T)^{\tilde{s}-1}+1) \left(\frac{|\mathcal{A}_T|}{n}\right)^{1/2} \right), \quad & \text{if $d_{\mu} \in (2 \tilde s-2,2\tilde s)$}.
	\end{cases}
	\]
\end{lemma}	
\begin{proof}
	Throughout the proof, $C$ is a generic constant that may change its value from line to line, but will never depend on $n$, $\varepsilon$ or $T$.
	For the proof, we first state some elementary properties:
	\begin{itemize}
		\item[(i)] By Lemma \ref{lem:decomposition}, $E_{\tilde{s}, g, T} \leq \|g\|_\infty \left(\frac{|\mathcal{A}|_T}{n}\right)^{1/2}$, where we simply use that $\|f\|_\infty \leq 1$ for all $f \in C^{\tilde s-1,1}_1$.
		\item[(ii)] By Lemma \ref{lem:decomposition},   $M_{j,g} \leq \lVert g \rVert_{\infty}\sqrt{\frac{|\mathcal{A}|_j}{n}}$, and hence for $k < d_{\mu}/2$,
		\[
		\sum_{j=i}^{T-1} D_{j+1}^k M_{j, g} \leq C \frac{\varepsilon^{k-d_{\mu}/2}}{n^{1/2}} \sum_{j=i}^{T-1} (3^{(k-d_{\mu}/2)j}) \leq C \frac{\varepsilon^{k-d_{\mu}/2}}{n^{1/2}}3^{(k-d_{\mu}/2)i},
		\]
		where we note that $\sum_{j=i}^{T-1} q^j \leq \frac{1}{1-q} q^i \leq C q^i$ for $q = 3^{k-d_{\mu}/2} < 1$.
		\item[(iii)] We have for $k=d_{\mu}/2$
		\[ 
		\sum_{j=i}^{T-1} D_{j+1}^k M_{j, g} \leq C \frac{T}{n^{1/2}},
		\]
		and for $k >d_{\mu}/2$
		\[ 
		\sum_{j=i}^{T-1} D_{j+1}^k M_{j, g} \leq C \frac{(3^T\varepsilon)^{k-d_{\mu}/2}}{n^{1/2}}.
		\]
	\end{itemize}
	
	These properties together with Lemma~\ref{lemma:estimateE} will be enough to yield the claim inductively over increasing $\tilde{s}$. Indeed, for $\tilde{s}=1$, we simply have 
	\begin{align*}
	E_{1, g, i}  \leq E_{1, g, T} + \sum_{j=i}^{T-1}   D_{j+1} M_{j, g},
	\end{align*}
	where the bound for the right hand side follows from property (i) and (ii) directly. 	
	
	For the induction step, assume $\tilde{s} <d_{\mu}/2$ and the claim is true for $1, \dots, \tilde{s}-1$. We build on the lemma \ref{lemma:estimateE}. The first term therein, $E_{\tilde{s}, g, T}$, can be controlled by Property (i) and is thus included in the inequality we want to show by possibly increasing $C$ by $1$. The term $\sum_{|\alpha|=\tilde{s}} \frac{1}{\alpha!}D_{j+1}^{\tilde s} M_{j, g}$ can be treated via Property (ii). To conclude, we only need to treat the terms for $1 \leq |\alpha|\leq \tilde{s}-1$ using the induction hypothesis. Using $D_{j+1} \lesssim 3^j \varepsilon$, we find that 
	\begin{align*}
	&\sum_{j=i}^{T-1} E_{\tilde s-|\alpha|, g_{D_{j+1}^{|\alpha|}}, j+1} \\
	&\leq C \|g\|_\infty \sum_{j=i}^{T-1} D_{j+1}^{|\alpha|} \left(3^{j(\tilde{s}-|\alpha|-\frac{d_{\mu}}{2})}\frac{\varepsilon^{\tilde{s}-|\alpha|-d_{\mu}/2}}{n^{1/2}} + ((\varepsilon 3^T)^{\tilde{s}-|\alpha|-1}+1) \left(\frac{|\mathcal{A}_T|}{n}\right)^{1/2} \right) \\
	&\leq C\|g\|_\infty \sum_{j=i}^{T-1} 3^{j |\alpha|} \varepsilon^{|\alpha|} \left(3^{j(\tilde{s}-|\alpha|-\frac{d_{\mu}}{2})}\frac{\varepsilon^{\tilde{s}-|\alpha|-d_{\mu}/2}}{n^{1/2}} + ((\varepsilon 3^T)^{\tilde{s}-|\alpha|-1}+1) \left(\frac{|\mathcal{A}_T|}{n}\right)^{1/2} \right) \\
	&\leq C\|g\|_\infty \left(\frac{\varepsilon^{\tilde{s}-d_{\mu}/2}}{n^{1/2}} \sum_{j=i}^{T-1} 3^{j(\tilde{s}-\frac{d_{\mu}}{2})} + \left(\frac{|\mathcal{A}_T|}{n}\right)^{1/2}  ((\varepsilon 3^T)^{\tilde{s}-|\alpha|-1}+1) \sum_{j=i}^{T-1} 3^{j |\alpha|} \varepsilon^{|\alpha|}  \right)\\
	&\leq C\|g\|_\infty \left(\frac{\varepsilon^{\tilde{s}-d_{\mu}/2}}{n^{1/2}} 3^{i(\tilde{s}-\frac{d_{\mu}}{2})} + \left(\frac{|\mathcal{A}_T|}{n}\right)^{1/2}  ((\varepsilon 3^T)^{\tilde{s}-1}+1)  \right),
	\end{align*}
	where the final inequality follows from the fact that $\sum_{j=i}^{T-1} 3^{j(\tilde{s}-\frac{d_{\mu}}{2})} \leq C3^{i(\tilde{s}-\frac{d_{\mu}}{2})}$ and $\sum_{j=i}^{T-1} 3^{j |\alpha|} \varepsilon^{|\alpha|}  \leq C (3^T\epsilon)^{|\alpha|}$. The induction and thus the proof for the case $d_{\mu}>2\tilde s$ are complete.
	
	Finally, only one of the two cases $d_{\mu}=2\tilde s, d_{\mu} \in (2\tilde s -2, 2\tilde s)$ could happen. We find that the term $\sum_{|\alpha|=\tilde{s}} \frac{1}{\alpha!}D_{j+1}^{\tilde s} M_{j, g}$ can be controlled by Property (iii), and similarly as in the last paragraph we get that 
	\begin{align*}
	&\sum_{j=i}^{T-1} E_{\tilde s-|\alpha|, g_{D_{j+1}^{|\alpha|}}, j+1} \\
	& \leq C\|g\|_\infty \left(\frac{\varepsilon^{\tilde{s}-d_{\mu}/2}}{n^{1/2}} \sum_{j=i}^{T-1} 3^{j(\tilde{s}-\frac{d_{\mu}}{2})} + \left(\frac{|\mathcal{A}_T|}{n}\right)^{1/2}  ((\varepsilon 3^T)^{\tilde{s}-1}+1)   \right) \\
	& \leq \begin{cases}
	             C\|g\|_\infty \left(\frac{T}{n^{1/2}}+\left(\frac{|\mathcal{A}_T|}{n}\right)^{1/2}  ((\varepsilon 3^T)^{\tilde{s}-1}+1) \right), \quad & \text{if $d_{\mu}=2\tilde s$}, \\
	             C\|g\|_\infty \left(\frac{(3^T\varepsilon)^{\tilde{s}-d_{\mu}/2}}{n^{1/2}}+\left(\frac{|\mathcal{A}_T|}{n}\right)^{1/2}  ((\varepsilon 3^T)^{\tilde{s}-1}+1) \right), \quad & \text{if $d_{\mu} \in (2\tilde s-2,2\tilde s)$}.
	          \end{cases}
	\end{align*}
	In conjugation with the estimate of $E_{\tilde s, g,T}$ from Property (i), we finish the induction step by Lemma~\ref{lemma:estimateE}.
\end{proof}

\begin{proof}[Proof of Theorem~\ref{thm:empiricalrate}]
	Thanks to Remark~\ref{rmk:dim}, it is sufficient to prove the result for $d_{\mu}>2s-2$. According to the definition of intrinsic dimension, we have that 
	\begin{align*}
	\mathbb{E} \left[\sup_{f \in C^{s-1,1}_1} \int f \, (\mu-\h \mu^n)(dx) \right] \leq & \mathbb{E} \left[\sup_{f \in C^{s-1,1}_1} \int_{\bar \Omega} f \, (\mu-\h \mu^n)(dx) \right] \\
	&+ \mathbb{E} \left[\sup_{f \in C^{s-1,1}_1} \int_{\bar \Omega^c} f \, (\mu-\h \mu^n)(dx) \right],
	\end{align*} 
	where the second term on the right hand side is simply bounded by $2 \mu(\bar \Omega^c) $ due to $\| f \|_{\infty}\leq 1$. 
	
	Recalling that $\bar \Omega = \cup_{A_1 \in \cA_1} A_1$, for any $x \in A_1$, we take the Taylor expansion of $f \in C^s_1$
	\begin{align*}
	f(x)=f(x_{A_1})+\sum_{1 \leq |\alpha| \leq { s}-1} \frac{D^{\alpha}f(x_{A_{1}})}{\alpha!}(x-x_{A_{1}})^{\alpha}+R^f(x,x_{A_1}),
	\end{align*} 
	where the last term is bounded by $D_1^s$. Now for each index $\alpha$, define $g^{\alpha}(x):=\sum_{A_1 \in \cA_1} \mathbbm{1}_{\{x \in A_1\}}(x-x_{A_1})^{\alpha}$ where $\| g^{\alpha} \|_{\infty} \leq D_1^{|\alpha|}$. Then taking the supremum over all $f$ in the equation above and taking expectation yield to that 
	\begin{align*}
	\mathbb{E}\left[\sup_{f \in C_1^{s-1,1}} \int_{\bar \Omega} f \,d(\mu - \mu^n)\right] \leq  E_{s, 1, 1} + \sum_{1 \leq |\alpha| \leq s-1} E_{s-|\alpha|, g^{\alpha}, 1}+ D_1^s.
	\end{align*}

	Let us first prove for the case $d_{\mu}>2s$. Invoking Lemma~\ref{lem:induction_s} and that $D_1 \leq 3\varepsilon$, we obtain that 
	\begin{align*}
	\mathbb{E}\left[\sup_{f \in C_1^{s-1,1}} \int_{\bar \Omega} f \,d(\mu - \mu^n)\right] \leq C \left(\varepsilon^s + \frac{\varepsilon^{s-d_{\mu}/2}}{n^{1/2}} + ((\varepsilon 3^T)^{s-1}+1) \left(\frac{|\mathcal{A}_T|}{n}\right)^{1/2}\right),
	\end{align*}
	and therefore
	\begin{align*}
	\mathbb{E}& \left[\sup_{f \in C_1^{s-1,1}} \int f \,d(\mu-\mu^n)\right] \\
	& \leq  C \left(\varepsilon^s + \frac{\varepsilon^{s-d_{\mu}/2}}{n^{1/2}} + ((\varepsilon 3^T)^{s-1}+1) \left(\frac{|\mathcal{A}_T|}{n}\right)^{1/2}\right) + 2 \mu(\bar{\Omega}^c).
	\end{align*}
	Then we choose $T = \log_3(\varepsilon^{-2s/d_{\mu}})$ and $\varepsilon = n^{-1/d_{\mu}}$. The first two terms in the estimate are still of order $n^{-s/d_{\mu}}$, it only remains to deal with the last two terms including $\mathcal{A}_T$ and $\bar{\Omega}$.  Note again that $\varepsilon 3^T \leq 1$, which means $((\varepsilon 3^T)^{s-1}+1) \leq 2$. Further, 
	\[|\mathcal{A}_T|/n \leq C (\varepsilon^{-2s/d_{\mu}} \varepsilon)^{-d_{\mu}}/(\varepsilon^{-d_{\mu}}) \leq C \varepsilon^{2s},
	\]
	and thus 
	\[
	((\varepsilon 3^T)^{s-1}+1) \left(\frac{|\mathcal{A}_T|}{n}\right)^{1/2} \leq C \varepsilon^{s} = C n^{-s/d_{\mu}}.
	\]
	And finally, 
	\[
	2 \mu(\bar{\Omega}^C) \leq C (3^T \varepsilon)^{(s\,d_{\mu})/(d_{\mu}-2s)} \leq C (\varepsilon^{1-2s/d_{\mu}})^{(s\,d_{\mu})/(d_{\mu}-2s)} = C \varepsilon^s = C n^{-s/d_{\mu}}.
	\]
	The overall estimate is complete.

	The argument for $d_{\mu} \leq 2s$ is almost the same. Again by Lemma~\ref{lem:induction_s}, we get that 
	\begin{align*}
	&\mathbb{E} \left[\sup_{f \in C_1^{s-1,1}} \int f \,d(\mu-\mu^n)\right] \\
	&\leq  
	\begin{cases} 
	C \left(\varepsilon^s + \frac{T}{n^{1/2}} + ((\varepsilon 3^T)^{s-1}+1) \left(\frac{|\mathcal{A}_T|}{n}\right)^{1/2}\right) \quad &\text{if $d_{\mu}=2s$}, \\
	C \left(\varepsilon^s + \frac{(3^T\varepsilon)^{s-d_{\mu}/2}}{n^{1/2}} + ((\varepsilon 3^T)^{s-1}+1) \left(\frac{|\mathcal{A}_T|}{n}\right)^{1/2}\right) \quad &\text{if $d_{\mu}\in (2s-2,2s)$}.
	\end{cases}
	\end{align*}
	To conclude, we simply choose $\epsilon= n^{-1/2s}$ and $T=\frac{1}{2s}\log(n) $.
\end{proof}

\section{Sample complexity of regularized optimal transport}\label{sec:samplecomplexity}
Throughout this section, we make the standing assumption that each marginal $\mu_i$ is supported on a compact subset $X_i$ of $\mathbb{R}^{d_i}$, $i=1,\dotso, N$. 
Further, we assume that $\varphi$ is dual regular.

Our first result quantifies the regularity of the optimal dual potentials depending on the regularity of $c$ and $\psi$.
\begin{proposition}[Regularity of optimal potentials]\label{prop:regularity}
	Let $s \in \mathbb{N}$ and assume $c \in C^s(X)$, $\psi \in C_{loc}^{s,1}(\mathbb{R})$. Then there exists a dual optimizer $h = \oplus_{i=1}^N h_i$ of $\OT_{\varphi}(\bs \mu)$ such that for all $i=1, \dots, N$, $h_i \in C^{s-1,1}_{K}$,
	where $K$ depends only on $\psi, c$ and the diameter of the space $X$.
	
\end{proposition}
\begin{proof}
	\emph{Step 1:} By our standing assumption that $\varphi$ is dual regular, we can take the dual optimizer $h$ from Lemma~\ref{lemma:dualoptimizer}, which we can normalize so that $\int h_i \,d\mu_i$ is constant in $i$ and thus $\sum_{i=1}^\infty \|h_i\|_\infty$ is bounded. Let us prove that $h_i$ is Lipschitz. Take two $x_i, \tilde{x}_i$, and assume without loss of generality that $h_i(x_i) \leq h_i(\tilde{x}_i)$. Due to the compactness of $X$, there exists a positive constant $L$ such that $|c(x_i,x^{-i})-c(\tilde{x}_i,x^{-i})| \leq L|x-\tilde{x}|$. Subsequently, we obtain that 
	\begin{align*}
	1 =& \int_{X^{-i}} \psi' ( h_i(\tilde x_i)+{h}^{-i}(x^{-i})-c(\tilde{x}_i,x^{-i})) \, P^{-i}(dx^{-i}) \\
	\leq &\int_{{X}^{-i}} \psi' ( h_i(x_i)+L|x_i-\tilde x_i| +h^{-i}(x^{-i})-c(\tilde x_i,x^{-i})) \, P^{-i}(dx^{-i}).
	\end{align*}
	Since $\psi'$ is strictly increasing over $[x_0-\delta,\infty)$, we conclude that $h_i(x_i) \leq h_i(\tilde{x}_i) \leq h_i(x_i)+L|x_i-\tilde x_i|$, i.e., $h_i$ is $L$-Lipschitz. 
	
	\vspace{5pt}
%
	\emph{Step 2:} 
	For each $i$, let us denote the defining function of $h_i$ by
		\begin{align}\label{eq:Fi}
		F_i(x_i,y_i)= \int_{X^{-i}} \psi'(y_i+ h^{-i}(x^{-i}) -c(x_i,x^{-i})) \, P^{-i}(dx^{-i}),
		\end{align}
		where $h_i$ satisfies $F_i(x_i,h_i(x_i))=1$ for all $ x_i \in X_i$. 
	We already proved that $h_i \in C^{0,1}_K$ for $i=1,\dotso ,N$. Supposing $s \geq 2$ and that $\{h_i\}_{i=1}^N \subset C_K^{k,1}$ is true for $k \geq 0$ with $k\leq s-2$, let us prove that $\{h_i\}_{i=1}^N \subset C_K^{k+1,1}$.

	Due to the compactness of $X^{-i}$ and our assumption, $F_i: X_i \times \mathbb{R} \to \mathbb{R}$ defined in \eqref{eq:Fi} is differentiable, and has first order derivatives 
	\begin{align}\label{eq:implicitfunctiontheorem}
	& \frac{\pa F_i}{\pa x_i}(x_i,y_i)=-\int_{{X}^{-i}} \psi''(y_i+ h^{-i}(x^{-i}) -c(x_i,x^{-i})) \pa_{x_i} c(x_i,x^{-i}) \, P^{-i}(dx^{-i}), \\
	&\frac{\pa F_i}{\pa y_i}(x_i,y_i)=\int_{{X}^{-i}} \psi''(y_i+ h^{-i}(x^{-i}) -c(x_i,x^{-i})) \, P^{-i}(dx^{-i}).
	\end{align}
	The defining equation $F_i(x_i, h_i(x_i)) = 1$ for $h_i$ thus yields, according to the implicit function theorem (c.f.~\cite[Theorem 9.28]{MR0385023}), that $h_i$ is differentiable and we have 
	\begin{align*}
	&\pa_{x_i} h_i(x_i) \int_{{X}^{-i}} \psi''(h_i(x_i)+ h^{-i}(x^{-i}) -c(x_i,x^{-i})) \, P^{-i}(dx^{-i})\\
	&= \int_{{X}^{-i}} \psi''(h_i(x_i)+ h^{-i}(x^{-i}) -c(x_i,x^{-i})) \pa_{x_i} c(x_i,x^{-i}) \, P^{-i}(dx^{-i}).
	\end{align*}
	Denote  
	\begin{align}
	f_i(x_i):&= \pa_{x_i} F_i(x_i, h_i(x_i)) \label{eq:fi}\\
	&=-\int_{{X}^{-i}} \psi''(h_i(x_i)+ h^{-i}(x^{-i}) -c(x_i,x^{-i})) \pa_{x_i} c(x_i,x^{-i}) \, P^{-i}(dx^{-i}),  \notag \\
	g_i(x_i):&= \pa_{y_i} F_i(x_i,h_i(x_i)) \label{eq:gi} \\
	&=\int_{{X}^{-i}} \psi''(h_i(x_i)+ h^{-i}(x^{-i}) -c(x_i,x^{-i})) \, P^{-i}(dx^{-i}) \notag, 
	\end{align}
	and hence \eqref{eq:implicitfunctiontheorem} can be rewritten as 
	\begin{align}\label{eq:hiequation}
	f_i(x_i)+ g_i(x_i) h^{(1)}_i(x_i)=0.
	\end{align}
	
Let us show that $g_i(x_i)$ is bounded from below by a positive constant. Define $$p(x_i):=\mathbb{P}^{-i}[\{x^{-i}: \, h_i(x_i)+ h^{-i}(x^{-i}) -c(x_i,x^{-i}) < x_0-\delta \}].$$ 
According to our assumption, there exists a positive constant $\alpha$ such that $\psi''(x)\geq \alpha$ for all $x \in [x_0-\delta,\infty)$. Therefore for all $x^{-i}$ such that $ h_i(x_i)+ h^{-i}(x^{-i}) -c(x_i,x^{-i}) < x_0-\delta$ we have $\psi''(h_i(x_i)+ h^{-i}(x^{-i}) -c(x_i,x^{-i})) \leq 1- \delta \alpha$. Now by the first order condition and the boundedness of $\sum_{i=1}^N \lVert h_i\rVert_{\infty}+\lVert c \rVert_{\infty}$, we obtain that 
\begin{align*}
1= & \int_{X^{-i}} \psi' ( h_i( x_i)+{h}^{-i}(x^{-i})-c(x_i,x^{-i})) \, P^{-i}(dx^{-i}) \\
\leq & p(x_i)(1-\delta \alpha) + (1-p(x_i))\psi'(2C_0).
\end{align*}
Therefore $1-p(x_i) \geq \frac{\delta \alpha}{\psi'(2C_0)-1+\delta \alpha}$, and hence $g_i(x_i) \geq \alpha(1-p(x_i))$ is uniformly bounded from below.

	Taking $(k-1)$-th derivative of \eqref{eq:hiequation} with respect to $x_i$ yields that
	\begin{align*}
	f_i^{(k-1)}(x_i) + \sum_{j=1}^{k-1} {{k-1}\choose{j}} g_i^{(j)}(x_i)  \, h_i^{(k-j)}(x_i)+ g_i(x_i) h_i^{(k)}(x_i)=0,
	\end{align*}
and thus 
	\begin{align}\label{eq:h(k).derivative}
	h_i^{(k)}(x_i)= -\frac{f_i^{(k-1)}(x_i) + \sum_{j=1}^{k-1} {{k-1}\choose{j}} g_i^{(j)}(x_i)  \, h_i^{(k-j)}(x_i)}{g_i(x_i)}.
	\end{align}
	
	From \eqref{eq:fi} and \eqref{eq:gi}, it can be seen that $f_i^{(k-1)}, g_i^{(k-1)}$ are differentiable, due to our induction hypothesis and the exchangeability of integral and derivative over compact sets. Therefore, $h_i^{(k)}$ is differentiable by \eqref{eq:h(k).derivative}, the formula given by 
	\begin{align*}
	h_i^{(k+1)}(x_i)= -\frac{f_i^{(k)}(x_i) + \sum_{j=1}^{k} {{k}\choose{j}} g_i^{(j)}(x_i)  \, h_i^{(k+1-j)}(x_i)}{g_i(x_i)}.
	\end{align*}
	Also we note that  $f_i^{(k)}, g_i^{(k)}$ are Lipschitz due to the fact that $\psi^{(j)}$ is Lipschitz for all $j=1, \dotso, s$. Then since $g_i$ is uniformly bounded away from zero and from above, one can easily obtain the Lipschitz property of $h_i^{(k+1)}$. Indeed, $h_i^{(k+1)}$ is differentiable almost everywhere, and is equal to 
	\begin{align*}
	h_i^{(k+2)}(x_i)=-\frac{f_i^{(k+1)}(x_i) + \sum_{j=1}^{k+1} {{k+1}\choose{j}} g_i^{(j)}(x_i)  \, h_i^{(k+2-j)}(x_i)}{g_i(x_i)},
	\end{align*}
	where $f_i^{(k+1)}$ and $g_i^{(k+1)}$ are weak derivatives of Lipschitz functions $f_i^{(k)}$ and  $g_i^{(k)}$.

	Noting that $|f_i^{(k+1)}(x_i) + \sum_{j=1}^k {{k+1}\choose{j}} g_i^{(j)}(x_i)  \, h_i^{(k+2-j)}(x_i)|$ is bounded only depending on $\lVert \psi^{(j+2)} \rVert_{\infty}$, $\lVert c^{((j+1)\vee 0) } \rVert_{\infty}$, $\lVert h^{(j\vee 0)}_i \lVert_{\infty}, j=-2,\dotso,k+1$,  together with the induction hypothesis we obtain $h_i \in C^{k+1,1}_K$. 
\end{proof}

The lemma below follows the idea by \cite{mena2019statistical}.
\begin{lemma}[Dual continuity of $\OT_{\varphi}$]\label{lemma:dualcon}
	Let $s \in \mathbb{N}$ and $h, \tilde{h}$ be dual optimizers for $\OT_{\varphi}(\bs\mu)$ and $\OT_{\varphi}(\tilde{\mu}_1, \mu_2, \dots, \mu_N)$, respectively. Assume $K > 0$ such that $(h - \psi(h - c)), (\tilde{h} - \psi(\tilde{h} - c)) \in C^{s-1,1}_{K}(X)$.
	Then
	\[
	\left| \OT_{\varphi}(\bs\mu)-\OT_{\varphi}(\tilde{\mu}_1, \mu_2, \dots, \mu_N) \right| \leq 3  \|\mu_1 - \tilde{\mu}_1\|_{C^{s-1,1}_{K}(X_1)}=3  K\|\mu_1 - \tilde{\mu}_1\|_{C^{s-1,1}_{1}(X_1)},
	\]
	where $\|\mu_1 - \tilde{\mu}_1\|_{C^{s-1,1}_{K}(X_1)}:= \sup_{f \in C^{s-1,1}_K(X_1)} \int f \, d(\mu_1-\tilde{\mu}_1)$. 
\end{lemma}
\begin{proof}
	Let $a := \int (h - \psi(h - c)) \,dP^{-1} \in C^{s-1}_{K}(X_1)$ and $b := \int (\tilde{h} - \psi(\tilde{h} - c)) \,dP^{-1} \in C^{s-1,1}_{K}(X_1)$. We calculate
	\begin{align*}
	&\left| \OT_{\varphi}(\bs\mu) - \OT_{\varphi}(\tilde{\mu}_1, \mu_2, \dots, \mu_N)\right|
	= \left| \int a \,d\mu_1 - \int b\,d\tilde{\mu}_1\right| \\
	&\leq \left| \int a \,d(\mu_1 - \tilde{\mu}_1) \right| + \left| \int (a-b) \,d\tilde{\mu}_1 \right| \leq \|\mu_1 - \tilde{\mu}_1\|_{C_{K}^{s-1,1}} + \int (b-a) \,d\tilde{\mu}_1,
	\end{align*}
	where the final term can be estimated via
	\[
	\int (b-a)\,d\tilde{\mu}_1 = \int (b-a) \,d(\tilde{\mu}_1 - \mu_1) + \int (b-a) \,d\mu_1 \leq 2 \|\mu_1 - \tilde{\mu}_1\|_{C_{K}^{s-1,1}(X_1)},
	\]
	since $\int (b-a) \,d\mu_1 \leq 0$. Putting the two inequalities together yields the claim.
\end{proof}

The following theorem is a combination of Theorem~\ref{thm:empiricalrate}, Proposition~\ref{prop:regularity}, and Lemma~\ref{lemma:dualcon}. We emphasize again that the compactness of space $X$, dual regularity of $\varphi$ and \eqref{eq:strictconvexconjugate} are standing assumptions, and recall the definition of $\psi$ from \eqref{eq:psidef}.
\begin{theorem}[Sample complexity of $\OT_{\varphi}$]\label{thm:samplecomplexity}
	Assume $\mu_1, \dots, \mu_N$ satisfy $I(d_\mu, s)$ for some $s \in \mathbb{N}$, $d_\mu \in (2s, \infty)$ and that $c \in C^s(X)$, $\psi \in C^{s,1}(\mathbb{R})$. Then we have 
	\[
	\mathbb{E}\left[ \left|\OT_{\varphi}(\bs\mu) - \OT_{\varphi}(\h\mu_1^n, \dots, \h \mu_N^n)\right| \right] \lesssim n^{-s/d_\mu}.
	\]
	\begin{proof}
		Note that our Proposition~\ref{prop:regularity} holds for any $(\mu_1,\dotso,\mu_N)$ with bounded support, and hence the assumption of Lemma~\ref{lemma:dualcon} is always satisfied. The difference $| \OT_{\varphi}(\mu_1,\dotso, \mu_N)- \OT_{\varphi}(\h \mu_1,\dotso, \h \mu_N)|$ is bounded from above by 
		\begin{align*}
		&\hspace{4.75mm}| \OT_{\varphi}(\mu_1,\mu_2, \dotso,\mu_N)- \OT_{\varphi}(\h \mu_1,\mu_2, \dotso, \mu_N)| \\
		&+| \OT_{\varphi}(\h \mu_1,\mu_2, \dotso,\mu_N)- \OT_{\varphi}(\h \mu_1,\h \mu_2, \dotso, \mu_N)|  \\
		&  \hspace{4cm} \vdots \\
		&+ | \OT_{\varphi}(\h \mu_1, \dotso,\h \mu_{N-1}, \mu_N)- \OT_{\varphi}(\h \mu_1, \dotso, \h \mu_{N-1}, \h \mu_N)|. 
		\end{align*} 
		Applying Lemma~\ref{lemma:dualcon} to every term in the above, we obtain that 
		\begin{align*}
		| \OT_{\varphi}(\mu_1,\dotso, \mu_N)- \OT_{\varphi}(\h \mu_1,\dotso, \h \mu_N)| \leq 3K \sum_{i=1}^N \| \mu_i - \h \mu_i \|_{C^{s-1,1}_1(X_i)}. 
		\end{align*}
		Taking expectation on both sides and using Theorem~\ref{thm:empiricalrate}, we conclude the result. 
	\end{proof}
\end{theorem}

	\begin{remark}\label{rmk:edependence}
		While the order $n^{-s/d_{\mu}}$ given by Theorem \ref{thm:samplecomplexity} is independent of the regularization parameter $\varepsilon$ (cf.~equation \eqref{eq:DOTintro}), it is nevertheless important to discuss the influence on the constants involved. One way of incorporating $\varepsilon$ is to study $\OT_{\varepsilon \varphi}$ and write the result from Theorem \ref{thm:empiricalrate} as
		\[
			\mathbb{E}\left[ \left|\OT_{\varepsilon\varphi}(\bs\mu) - \OT_{\varepsilon \varphi}(\h\mu_1^n, \dots, \h \mu_N^n)\right| \right] \leq C_\varepsilon n^{-s/d_\mu}.
		\]
		In this work (similar to \cite{genevay2019sample} in the entropic case), the constant $C_\varepsilon$ arises through $L^\infty$-bounds on the derivatives of the dual potentials shown in Proposition \ref{prop:regularity}. In the entropic case, it is well known that these bounds scale like $\|h^{(k)}\|_\infty \lesssim \varepsilon^{-(k-1)}$, see for instance \cite[Theorem 2]{genevay2019sample}. However, to arrive at the constant $C_\varepsilon$, we also need to take the term $\psi(h-c)$ in the dual into account, which is correspondingly exponential in the entropic case, and thus \cite[Theorem 3]{genevay2019sample} obtains a constant $C_\varepsilon \sim \varepsilon^{-\lfloor d/2\rfloor} \exp(\kappa/\varepsilon)$ for some $\kappa > 0$ in the entropic case.
		
		In the case of general divergences, obtaining explicit bounds for $\|h^{(k)}\|_\infty$ is much more difficult than in the entropic case. The reasons are twofold: Firstly, the entropic case is essentially the only one where we can solve the first order equations for $h_i$ explicitly, and secondly, the method of proof in Proposition \ref{prop:regularity} depends on a strict convexity estimate for $\psi$, which cannot be made quantitative without further assumptions. Nevertheless, one very relevant case beyond the entropic one, which is the one of polynomial divergences, can be made explicit, see Lemma \ref{lem:explicitepsilon}. There, we show that for $\psi_{\varepsilon}(y) = \frac{(y_+)^\beta}{\varepsilon^{\beta-1}}$ for $\beta \in \mathbb{N}_{\geq 2}$, we get the order $\|h^{(k)}\|_{\infty} \lesssim \varepsilon^{-2(k-1)}$, which then leads to an overall bound of order $C_\varepsilon \sim \varepsilon^{-3\beta + 5}$. This means that while the bounds for $\|h^{(k)}\|_\infty$ are worse compared to the entropic case, the overall dependence is still improved compared to the entropic case, since the influence of $\psi$ is no longer exponential.
		
		
	\end{remark}

\appendix
\section{Scaling in $\varepsilon$}
	\begin{lemma}\label{lem:explicitepsilon}
		Assume $\psi_{\varepsilon}(y) = \frac{(y_+)^\beta}{\varepsilon^{\beta-1}}$ for some $\beta \in \mathbb{N}$, $\beta \geq 2$, is the dual of $\varepsilon \varphi$ satisfying Definition \ref{assume1}. Let $c \in C^{\beta-1}$ and let $X$ be compact. Let $h_i^\varepsilon$ be the optimal dual potentials satisfying the first order conditions of Lemma \ref{lemma:dualoptimizer}, i.e.,
		\begin{align*}
		1= \beta \int_{X^{-i}} \left(\frac{h^{\varepsilon}(x)-c(x)}{\varepsilon} \right)_{+}^{\beta-1} \, P^{-i}(dx^{-i}), \quad \text{for all $x_i \in X_i$, $i=1, \dots, N$}.
		\end{align*}
		Then, there exists $C > 0$ independent of $\varepsilon$ such that for $k\leq \beta-2$, we have $\|h_i^{\varepsilon, (k)}\|_\infty \leq C \varepsilon^{-2(k-1)}$ and $h_i^{(\beta-2)}$ is Lipschitz continuous with constant bounded by $C \varepsilon^{-2(\beta-2)}$. Therefore we have that 
		\begin{align*}
		\lVert h^{\varepsilon}-\psi_{\varepsilon}(h^{\varepsilon}-c) \rVert_{C^{\beta-2,1}} \leq \frac{C}{\varepsilon^{3\beta-5}}
		\end{align*}
		\begin{proof}
			
			As in \emph{Step 1} of Proposition~\ref{prop:regularity}, it can be seen that $\{h^{\varepsilon}_i\}_{i=1,\dotso,N}$ are $L$-Lipschitz, where $L$ is the Lipschitz constant of the cost function $c$ independent of the scaling $\varepsilon$, and hence $\lVert h_i^{\varepsilon,(1)}\rVert_{\infty} \leq L$, $\sum_{i=1}^N \lVert h^{\varepsilon}_i \rVert_{\infty} \leq C_0$. Together with the first order condition 
			\begin{align*}
			1= \beta \int_{X^{-i}} \left(\frac{h^{\varepsilon}(x)-c(x)}{\varepsilon} \right)_{+}^{\beta-1} \, P^{-i}(dx^{-i}), \quad \text{for all $x_i$},
			\end{align*}
			we have the estimate
			\begin{align*}
			g^{\varepsilon}_i(x_i)=& \frac{\beta(\beta-1)}{\varepsilon} \int_{X^{-i}} \left(\frac{h^{\varepsilon}(x)-c(x)}{\varepsilon} \right)_{+}^{\beta-2} \, P^{-i}(dx^{-i}) \\
			\geq & \frac{\beta(\beta-1)}{2C_0} \int_{X^{-i}} \left(\frac{h^{\varepsilon}(x)-c(x)}{\varepsilon} \right)_{+}^{\beta-1} \, P^{-i}(dx^{-i})=\frac{\beta(\beta-1)}{2C_0},
			\end{align*}
			and thus $g^{\varepsilon}_i$ has a uniform lower bound independent of $\varepsilon$. Also it follows by H\"{o}lder's inequality that 
			\begin{align*}
			g_i^{\varepsilon,(1)}(x_i)=&\frac{\beta(\beta-1)(\beta-2)}{\varepsilon^2}\int_{X^{-i}} \left(\frac{h^{\varepsilon}(x)-c(x)}{\varepsilon} \right)_{+}^{\beta-3} (h_i^{\varepsilon,(1)}-\partial_{x_i}c(x))\, P^{-i}(dx^{-i}) \\
			\leq& \frac{\beta(\beta-1)(\beta-2)\lVert h_i^{\varepsilon,(1)}\rVert_{\infty}}{\varepsilon^2}  \left(\int_{X^{-i}} \left(\frac{h^{\varepsilon}(x)-c(x)}{\varepsilon} \right)_{+}^{\beta-1} \, P^{-i}(dx^{-i}) \right)^{\frac{\beta-3}{\beta-1}}  \leq \frac{C}{\varepsilon^2}.
			\end{align*}
			Similarly we have $f_i^{\varepsilon,(1)}(x_i) \leq C/\varepsilon^2$, and thus by \eqref{eq:h(k).derivative}, we also have \[|h^{\varepsilon,(2)}_i(x_i)| = |f^{\varepsilon,(1)}_i(x_i)/g_i^{\varepsilon}(x_i) + g^{\varepsilon,(1)}_i(x_i)h^{\varepsilon,(1)}_i(x_i) / g^{\varepsilon}_i(x_i)| \leq C/\varepsilon^2.\]
			Differentiating $g_i^{\varepsilon,(1)}(x_i)$, we obtain that $g_i^{\varepsilon, (2)}(x_i) \leq C \|h_i^{\varepsilon, (2)}\|_\infty/\varepsilon^2 \leq C/\varepsilon^{4}$,
			and the same bound for $f_i^{\varepsilon,(2)}$. Therefore, we have 
			\begin{align*}
			\left| h_i^{\varepsilon,(3)}(x_i) \right|= \left| \frac{f_i^{\varepsilon,(2)}(x_i)+g_i^{\varepsilon,(2)}(x_i)h_i^{\varepsilon,(1)}(x_i)+g_i^{\varepsilon,(1)}(x_i)h_i^{\varepsilon,(2)}(x_i)}{g_i^{\varepsilon}(x_i)} \right| \leq \frac{C}{\varepsilon^4}. 
			\end{align*}
			By induction, thanks to the explicit formula of $g_i^{\varepsilon}, f_i^{\varepsilon}$, it can be seen that for all $k=1,\dotso,\beta-2$,
			\begin{align*}
			\left| f_i^{\varepsilon,(k)}(x_i)\right|+\left| g_i^{\varepsilon,(k)}(x_i)\right| \leq \frac{C\lVert h_i^{\varepsilon,(k)} \rVert_{\infty}}{\varepsilon^2} \leq \frac{C}{\varepsilon^{2k}},
			\end{align*}	
			and according to \eqref{eq:h(k).derivative}
			\begin{align*}
			|h_i^{\varepsilon,(k+1)}(x_i)| \leq \frac{C}{\epsilon^{2k}}.
			\end{align*}

			To finalize, using the chain rule one can obtain upper bounds for partial derivatives of $\psi_{\varepsilon}(h^{\varepsilon}(x)-c(x))$,  
			\begin{align*}
			\left| D^{\alpha} \psi_{\varepsilon}(h^{\varepsilon}(x)-c(x)) \right| \leq \frac{C}{\varepsilon^{2(k-1)+\beta-1}}, \quad \text{for all $\alpha \in \mathbb{N}_0^d$ with $|\alpha| \leq k=1,\dotso,\beta-1$}. 
			\end{align*}
			
		\end{proof}
\end{lemma}

\begin{acks}[Acknowledgments]
The authors thank Daniel Bartl and Beno\^it Kloeckner for helpful comments. We are further grateful to an anonymous reviewer and the associate editor for their thorough reading and careful feedback which greatly improved the paper.
\end{acks}

\begin{funding}
Erhan Bayraktar is supported in part by the National Science Foundation under DMS-2106556, and in part by the Susan M. Smith Professorship.
\end{funding}

\bibliographystyle{imsart-nameyear} 
\bibliography{ref_ber}       

%
%
%
%
%

\end{document}